\documentclass[reqno]{amsart}
\usepackage{amsfonts,amsmath,amssymb,amsrefs,dsfont}
\usepackage{color}
\usepackage[utf8]{inputenc}  
\usepackage[T1]{fontenc}  
\usepackage{graphicx} 
\numberwithin{equation}{section}

\usepackage[colorinlistoftodos]{todonotes}

\newtheorem{thm}{Theorem}[section]
\newtheorem*{thm*}{Theorem}
\newtheorem{lem}{Lemma}[section]

\newtheorem*{cor*}{Corollary}

\newtheorem{Step}{Step}[section]

\newcommand{\R}{\mathbb{R}}
\newcommand{\bra}[1]{\left({#1}\right)}
\newcommand{\ph}{\varphi}
\newcommand{\eps}{\varepsilon}

\begin{document}
\title[Nonzero limiting profile for a critical Moser-Trudinger equation]{Glueing a peak to a non-zero limiting profile for a critical Moser-Trudinger equation}

\author{Gabriele Mancini}
\address[Gabriele Mancini]{Universit\`a degli Studi di Padova, Dipartimento di Matematica Tullio Levi-Civita, Via Trieste, 63, 35121 Padova}
\email{gabriele.mancini@math.unipd.it}

\author{Pierre-Damien Thizy}
\address[Pierre-Damien Thizy]{Universit\'e Claude Bernard Lyon 1, CNRS UMR 5208, Institut Camille Jordan, 43 blvd. du 11 novembre 1918, F-69622 Villeurbanne cedex, France}
\email{pierre-damien.thizy@univ-lyon1.fr}
\date{December 2017}

\begin{abstract}
Druet \cite{DruetDuke} proved that if $(f_\gamma)_\gamma$ is a sequence of Moser-Trudinger type nonlinearities with {critical growth}, and if $(u_\gamma)_\gamma$ solves \eqref{DruetDukeEq} and converges weakly in $H^1_0$ to some $u_\infty$, then the Dirichlet energy is quantified, namely there exists an integer $N\ge 0$ such that the energy of $u_\gamma$ converges to $4\pi N$ plus the Dirichlet energy of $u_\infty$. As a crucial step to get the general existence results of \cite{DruThiII}, it was more recently proved in \cite{DruThiI} that, for a specific class of nonlinearities (see \eqref{TypicalCase}),  the loss of compactness (i.e. $N>0$) implies that $u_\infty\equiv 0$. In contrast,  we prove here that there exist sequences $(f_\gamma)_\gamma$ of Moser-Trudinger type nonlinearities which admit a noncompact sequence  $(u_\gamma)_\gamma$ of solutions of \eqref{DruetDukeEq} having a nontrivial weak limit. 
\end{abstract}

\maketitle

\section*{Introduction}
Let $\Omega$ be a smooth bounded domain in $\mathbb{R}^2$. In this work we are interested in investigating the behavior of sequences of solutions of nonlinear critical elliptic problems of the form  \begin{equation}\label{DruetDukeEq}
\begin{cases}
&\Delta u  =f_\gamma(x,u)\,,~~ u>0\text{ in }\Omega\,,\\
&u =0\text{ on }\partial\Omega\,,
\end{cases}
\end{equation} 
 where $\Delta=-\text{div}(\nabla\cdot)$ and $(f_\gamma)_\gamma$  is a sequence of Moser-Trudinger type nonlinearities. A typical but very specific example of such a sequence is given by 
\begin{equation}\label{TypicalCase}
\begin{cases}
&f_\gamma(x,u)=\beta_\gamma h_\gamma(x) u \exp(u^2)\,,\\
&\lim_{\gamma\to +\infty}\beta_\gamma=\beta_\infty\ge 0\,,\\
&\lim_{\gamma\to +\infty}h_\gamma=h_\infty \text{ in }C^2(\bar{\Omega})\,,\\
&h_\infty>0 \text{ in }\bar{\Omega}\,,
\end{cases}
\end{equation}
where the $\beta_\gamma$'s are positive numbers and the $h_\gamma$'s are positive functions in $C^2(\bar{\Omega})$. Recall that nonlinearities as in \eqref{TypicalCase} arise when looking for critical points of the Moser-Trudinger functional {$$F_\gamma(u)=\int_{\Omega} h_\gamma(x) \exp(u^2) dx\,,
$$ }
under the Dirichlet energy constraint $\int_{\Omega} |\nabla u|^2 dx=\alpha_\gamma$, where $(\alpha_\gamma)_\gamma$ is any given sequence of positive real numbers.

 In \cite{DruetDuke} Druet  obtained a general quantification result for solutions of \eqref{DruetDukeEq} for a large class of nonlinearities, including the ones in \eqref{TypicalCase}. More precisely, he proved  that, if the $f_\gamma$'s  have \emph{uniformly critical growth} (see \cite[Definition 1]{DruetDuke}), then for any sequence $(u_\gamma)_\gamma$ satisfying \eqref{DruetDukeEq} and converging weakly in $H^1_0$ to some $u_\infty$, the Dirichlet energy is quantified (see also \cite[Section 2]{DruThiI}). Namely, up to a subsequence, there exists an integer $N\ge 0$ such that 
\begin{equation}\label{QuantifDru}
\lim_{\gamma\to +\infty}\int_{\Omega} |\nabla u_\gamma|^2 dy=4\pi N+\int_\Omega |\nabla u_\infty|^2 dy\,.
\end{equation}
Observe that such a sequence $(u_\gamma)_\gamma$ is compact in $H^1_0$, if and only if it is uniformly bounded, and if and only if $N=0$ in \eqref{QuantifDru}. As a consequence of the \textit{very strong} interaction generated by an exponentially critical nonlinearity, it is not clear in general whether it is possible to have loss of compactness, i.e. $N\ge 1$ in \eqref{QuantifDru}, together with $u_\infty\not \equiv 0$. For instance, for the typical nonlinearities $f_\gamma$ given by \eqref{TypicalCase}, in order to understand globally the bifurcation diagram of \eqref{DruetDukeEq} and the associated questions of existence of solutions (see \cite{DruThiII}), Druet-Thizy \cite{DruThiI} pushed further the analysis  and proved that, for such a \textit{noncompact} sequence $(u_\gamma)_\gamma$, there necessarily holds that $u_\infty\equiv 0$, so that the limit of the Dirichlet energy  in \eqref{QuantifDru} has to be $4\pi$ times an integer $N>0$. In contrast, the purpose of this paper is to show that for different families of exponentially critical nonlinearities it is possible to construct bubbling sequences of solutions with non-trivial weak limit in $H^1_0$.  

Let $\Omega$ be the unit disk of $\mathbb R^2$ centered at $0$ and let $0<\lambda_1<\lambda_2<...$ be the sequence of the simple radial eigenvalues of $\Delta$ in $\Omega$, with zero Dirichlet boundary condition.  Let $v_1,v_2...$ be the associated radial eigenfunctions uniquely determined by $\int_{\Omega} v_k^2 dy=1$ and $v_k(0)>0$ for all $k$. Our first result shows that  there exists a sequence $(f_\gamma)_\gamma$ with Moser-Trudinger type growth for which \eqref{DruetDukeEq} admits a  sequence $(u_\gamma)_\gamma$ of positive radial solutions  converging weakly to a multiple of $v_1$ and with Dirichlet energy approaching  any fixed value in $(4\pi,+\infty)$. 
\begin{thm}[Positive case]\label{MainThm}
Let $l>0$ be given and let $\Omega\subset\mathbb{R}^2$ be the unit disk centered at $0$. Let $\bar{\lambda}_{\gamma}$ be given by
\begin{equation}\label{DefLambdaGamma}
\bar{\lambda}_{\gamma}=\lambda_1-\varepsilon_\gamma\,,
\end{equation}
where $\varepsilon_\gamma=\frac{4\pi v_1(0) }{\gamma}\sqrt{\frac{\lambda_1}{l}}$. Then there exists $\beta_\gamma>0$ such that the equation
\begin{equation}\label{EqSingGamma}
\begin{cases}
&\Delta u=\bar{\lambda}_{\gamma} u+ \beta_\gamma u \exp(u^2)\,,~~ u>0\text{ in }\Omega\,,\\
&u=0\text{ on }\partial\Omega\,,
\end{cases}
\end{equation}
admits a smooth solution $u_\gamma$ satisfying $u_\gamma(0)=\gamma$, for all $\gamma>0$ sufficiently large for $\bar{\lambda}_\gamma$ to be positive. Moreover, for any sequences $(\beta_\gamma)_\gamma$,  $(u_\gamma)_\gamma$ with the above properties, we have that $\beta_\gamma\to 0$, that 
\begin{equation}\label{WeakConv}
u_\gamma\rightharpoonup u_\infty\text{ weakly in }H^1_0\,,
\end{equation}
where $u_\infty=v_1\sqrt{\frac{l}{\lambda_1}}\not\equiv 0$, and that the quantification 
\begin{equation}\label{QuantifWL}
\int_\Omega |\nabla u_\gamma|^2 dy=4\pi+l+o(1)\,,
\end{equation}
holds true as $\gamma\to +\infty$.
\end{thm}
Observe that \eqref{EqSingGamma} can be seen as a particular case of \eqref{DruetDukeEq} with 
\begin{equation}\label{EigenNonlin}
f_\gamma(x,u)=\bar \lambda_{\gamma} u + \beta_\gamma u \exp(u^2),
\end{equation}
and it arises when looking at the Euler-Lagrange equation of the Adimurthi-Druet inequality \cite{AdimurthiDruet}. We also refer to \cite{ThiMan}, where this Euler-Lagrange equation was studied in this tricky regime $\bar{\lambda}_{\gamma}\to \lambda_1$, but only in the minimal energy case, where $l$ equals $0$ in \eqref{QuantifWL}. When considering the typical case \eqref{TypicalCase}, i.e. the Euler-Lagrange equation of the standard Moser-Trudinger inequality, existence results have been obtained using radial analysis \cites{CarlesonChang,MalchMartJEMS} (see also \cite{MartMan}), variational \cites{Flucher,StruweCrit}, perturbative \cite{DelPNewSol} or topological methods \cites{DruThiII,LammRobertStruwe,StruwePlanar}. According to the previous discussion, contrary to those built in Theorem \ref{MainThm}, the blow-up solutions obtained in  these results always have a zero weak limit in $H^1_0$. 

As a by-product of Theorem \ref{MainThm} and its proof below, we also obtain the following result.
\begin{thm}[Nodal case]\label{MainThmNod}
Let $l>0$ be fixed and $\Omega\subset\mathbb{R}^2$ be the unit disk centered at $0$. Let $k\ge 2$ be a fixed integer. Then there exist positive real numbers $\bar{\gamma}=\bar{\gamma}(k,l)$, $\bar{\lambda}_{\gamma}=\bar{\lambda}_{\gamma}(k,l)$ and $\beta_{\gamma}=\beta_{\gamma}(k,l)$ such that  \begin{equation}\label{relEpsGamma}
\lambda_k-\bar{\lambda}_{\gamma}=(1+o(1))\frac{4\pi v_k(0)}{\gamma}\sqrt{\frac{\lambda_k}{l}},
\end{equation} 
as $\gamma\to +\infty$, and such that the equation 
\begin{equation}\label{EqSingGammaNod}
\begin{cases}
&\Delta u=\bar{\lambda}_{\gamma} u+ \beta_{\gamma} u \exp(u^2)\text{ in }\Omega\,,\\
&u=0\text{ on }\partial\Omega\,,\\
\end{cases}
\end{equation}
admits a smooth solution $u_{\gamma}=u_{\gamma}(k,l)$ satisfying $u_{\gamma}(0)=\gamma$, for all $\gamma>\bar{\gamma}$. Moreover, we have that $\beta_{\gamma}\to 0$, that \eqref{WeakConv} holds true for $$u_\infty=v_k\sqrt{\frac{l}{\lambda_k}}\not\equiv 0\,,$$ and that the quantification \eqref{QuantifWL} holds true, as $\gamma\to +\infty$.
\end{thm}

As $v_k$ for $k\ge 2$,  the solutions $u_\gamma$'s in Theorem \ref{MainThmNod} are sign-changing and have exactly $k$ nodal regions in $\Omega$. Theorem \ref{MainThmNod} provides new examples of non-compact sequences of nodal solutions for a Moser-Trudinger critical type equation for which the quantification in \eqref{QuantifWL} holds true. We mention that Grossi and Naimen \cite{GrossiNaim} obtained recently a nice example of a quantized sequence in the sign-changing case. In \cite{GrossiNaim}, the results of  \cites{AdYadNonex,AdYadMult} are used as a starting point and the point of view is completely different from that of Theorem \ref{MainThmNod}. 

While the nonlinearities of the form \eqref{EigenNonlin}  clearly have Moser-Trudinger type growth, it should be pointed out that they do not have \emph{uniformly critical growth} in the sense of the definition of Druet \cite{DruetDuke}, if $\beta_\gamma\to 0$ and $\bar \lambda_\gamma\to \lambda_k^-$, $k\ge 1$ as $\gamma\to +\infty$ as in Theorem  \ref{MainThm} and Theorem \ref{MainThmNod}. However, our techniques can be applied to the study of different kinds of  nolinearities.  For a fixed real parameter $a>0$, let  $g:[0,+\infty)\rightarrow \mathbb R$ be such that 
\begin{equation}\tag{G1}\label{G1}
g(t)= e^{t^2-a t} \quad \text{ for } \quad t\ge c_0>0,
\end{equation}
and 
{\begin{equation}\tag{G2}\label{G2}
g\in C^0([0,+\infty)) \quad  \text{with} \quad g> 0 \quad \text{ in  } [0,+\infty).
\end{equation}}
Then, we get the following result.

\begin{thm}\label{Thm3}
Let $\Omega{\subset \mathbb{R}^2}$ be the unit disk centered at $0$ and let $a>0$ and $g\in C{^0}([0,+\infty))$ be given so that \eqref{G1} and \eqref{G2} hold. For any $\gamma>0$ there exists a unique $\beta_\gamma>0$ such that the equation
\begin{equation}
\begin{cases}\label{LastProblem}
&\Delta u= \beta_\gamma u g(u)\,,~~ u>0\text{ in }\Omega\,,\\
&u=0\text{ on }\partial\Omega\,,
\end{cases}
\end{equation}
admits a (unique) smooth radially symmetric solution $u_\gamma$ satisfying $u_\gamma(0)=\gamma$. Moreover, as $\gamma\to + \infty$ we have $\beta_\gamma\to \beta_\frac{a}{2}>0$, that 
\begin{equation}\label{WLimThm3}
u_\gamma\rightharpoonup u_\frac{a}{2} \text{ weakly in }H^1_0\,,
\end{equation}
and that the following quantification holds true
\begin{equation}\label{QuantifGa}
\int_\Omega |\nabla u_\gamma|^2 dy=4\pi+\int_{\Omega}|\nabla u_\frac{a}{2}|^2 dy.
\end{equation}
\end{thm}

It is interesting to notice that the value at $0$ of the weak limit of the sequence $(u_\gamma)_\gamma$ in Theorem \ref{Thm3} depends only on the choice of $a$, that is on the asymptotic behavior of  the function $g$.  We stress that for any $a>0$ one can easily construct a function $g$ satisfying \eqref{G1} and \eqref{G2} such that $g(t)= 1$ for $t\le \frac{a}{2}$. For such function $g$, the nonlinearities  $f_\gamma(x,u)= \beta_\gamma u g(u)$ have \emph{uniformly critical growth} according to Druet's definition in \cite{DruetDuke}. Moreover, since the $u_\gamma$'s are positive and radially decreasing, one has $\beta_\frac{a}{2}=\lambda_1$ and $u_\frac{a}{2}=\frac{a}{2 v_1(0)} v_1$ so that the quantification result of  Theorem \ref{Thm3} reads as 
\begin{equation}\label{QuantificationImproved}
\int_{\Omega}|\nabla u_\gamma |^2dy\to 4\pi+ \frac{a^2\lambda_1}{4 v_1(0)^2}.
\end{equation}
Note that the value in the RHS of \eqref{QuantificationImproved} can be arbitrarily large or arbitrarily close to $4\pi$ depending on the choice of $a$.


\section*{Acknowledgements} 
The first author was supported by Swiss National Science Foundation, projects nr. PP00P2-144669 and PP00P2-170588/1. The second author was supported by BQR Accueil EC 2018, funding the new researchers in the university of Lyon 1.

\section{Proof of Theorem \ref{MainThm}}\label{SectPfThm1PositiveCase}
In the whole paper, $\Omega=B_0(1)$ is the unit disk centered at $0$ in $\mathbb{R}^2$. If $f$ is a radially symmetric function, since no confusion is possible, we will often write 
\begin{equation}\label{NotationRadial}
f(|x|)\text{ instead of }f(x)\,.
\end{equation} 
In the sequel, we let $\lambda_1$ and $v_1$ be  as in Theorem \ref{MainThm}. For all $R>0$, it is known that the smallest eigenvalue $\tilde{\lambda}_R$ of $\Delta$ in $B_0(R)$ with zero Dirichlet boundary condition is given by
\begin{equation}\label{Lambda1DeR}
\tilde{\lambda}_R=\frac{\lambda_1}{R^2}\,.
\end{equation}
\noindent Let $\tilde{v}_R\ge 0$ be the (radial) eigenfunction associated to ${\tilde{\lambda}}_R$  such that $\int_\Omega \tilde{v}_R^2 dy=1$.  We get first the following existence result. 
 
\begin{lem}\label{ExistRes}
Let $\lambda\in(0,\lambda_1)$ and $\gamma>0$ be given. Then there exists $\beta>0$ and a smooth function $u$ in $\Omega$ such that $u(0)=\gamma$ and such that $u$ solves 
\begin{equation}\label{EqSing}
\begin{cases}
&\Delta u=\lambda u+ \beta u \exp(u^2)\,,~~ u>0\text{ in }\Omega\,,\\
&u=0\text{ on }\partial\Omega\,.
\end{cases}
\end{equation}
\end{lem}

\begin{proof}[Proof of Lemma \ref{ExistRes}] 
Let $\lambda\in (0,\lambda_1)$ and $\gamma>0$ be fixed. For all given $\beta\ge0$, there exists a unique smooth and radially symmetric solution $u_\beta$ of
\begin{equation}\label{RadEqSing}
\begin{cases}
\Delta u=\lambda u+\beta u \exp(u^2)\,,\\
u(0)=\gamma\,,
\end{cases}
\end{equation}
well defined in $\mathbb{R}^2$. Indeed, if $[0,T_\beta)$ is the maximal interval of existence for $u_\beta$, we know from rather standard theory of radial elliptic equations that $T_\beta\in(0,+\infty]$, and that $T_\beta<+\infty$ implies that
\begin{equation}\label{RadCrit}
\limsup_{t\to T_\beta^-}|u_\beta'(t)|+|u_\beta(t)|=+\infty\,.
\end{equation}
Now observe that $E(u_\beta):[0,T_\beta)\to \mathbb{R}$ given by $E(u_\beta)=(u_\beta')^2+\lambda u_\beta^2+\beta \exp(u_\beta^2)$ is nonincreasing in $(0,T_\beta)$ by \eqref{RadEqSing}. Then, \eqref{RadCrit} cannot occur and $T_\beta=+\infty$ as claimed. 

Now, given $\beta>0$, there exists an $\varepsilon_\beta>0$ such that $u_\beta(s)> 0$ for all $s\in [0,\varepsilon_\beta]$ by continuity. Then, since the RHS of the first equation of \eqref{RadEqSing} is positive for $u_\beta>0$, and since 
 \begin{equation}\label{IPP}
- r u_\beta'(r)=\int_0^r (\Delta u_\beta)(s) s ds\,, 
\end{equation}
we get that
\begin{equation}\label{IntermComput}
u_\beta(r)\le \gamma+\log\frac{\varepsilon_\beta}{r} \int_0^{\varepsilon_\beta} (\Delta u_\beta)(s) s ds\,,
\end{equation}
for all $r\ge \varepsilon_\beta$ such that $u_\beta>0$ in $[\varepsilon_\beta,r]$. Setting now
\begin{equation}\label{DefRBeta}
R_\beta=\sup\left\{r>0\text{ s.t. }u_\beta(s)>0 \; \text{ for } s\in [0,r] \right\}\,,
\end{equation}
we clearly get from \eqref{IntermComput} that $0<R_\beta<+\infty$. Then we have  $u_\beta(R_\beta)=0$ and $u_\beta>0$ in $[0,R_\beta)$. Now, since $(\beta,r)\mapsto u_\beta(r)$ is continuous, and since we have that $u_\beta'(R_\beta)<0$ by \eqref{IPP}, we get that 
\begin{equation}\label{Continuity}
\beta\mapsto R_\beta \text{ is continuous in }[0,+\infty)\,.
\end{equation}
It is clear that $R_0=\sqrt{\lambda_1/\lambda}>1$   by \eqref{Lambda1DeR} and \eqref{RadEqSing}. Independently, multiplying \eqref{RadEqSing} by $\tilde{v}_{R_\beta}>0$  and integrating by parts in $B_0(R_\beta)$, we get that
\begin{equation*}
\begin{split}
{\tilde{\lambda}}_{R_\beta}\int_{B_0(R_\beta)} u_\beta {\tilde{v}}_{R_\beta} dy ~&=\int_{B_0(R_\beta)} \nabla u_\beta \nabla {\tilde{v}}_{R_\beta} dy\,,\\
&=\int_{B_0(R_\beta)} (\lambda+\beta \exp(u_\beta^2)) u_\beta {\tilde{v}}_{R_\beta} dy\,,\\
&> \beta\int_{B_0(R_\beta)} u_\beta {\tilde{v}}_{R_\beta} dy\,, 
\end{split}
\end{equation*}
so that ${\tilde{\lambda}}_{R_\beta} \to +\infty$ and then that $R_\beta\to 0$ as $\beta\to +\infty$, using \eqref{Lambda1DeR}. Thus by \eqref{Continuity}, we get that there exists $\beta>0$ such that $R_{\beta}=1$, which concludes the proof of Lemma \ref{ExistRes}. 
\end{proof}

\begin{proof}[Proof of Theorem \ref{MainThm}] Let $l>0$ be a {fixed real number}. Let $\bar{\lambda}_{\gamma}, \lambda_1$ and $v_1>0$ in $\Omega$ be as in the statement of Theorem \ref{MainThm}. By Lemma \ref{ExistRes}, there exists $\beta_\gamma>0$ such that the equation \eqref{EqSingGamma} admits a smooth solution $u_\gamma$ satisfying $u_\gamma(0)=\gamma$, for all $\gamma>0$ sufficiently large for $\bar{\lambda}_{\gamma}$ to be positive. First, we check that 
\begin{equation}\label{EqBetaToZero}
 \beta_\gamma=O\left(\frac{1}{\gamma} \right)\,,
\end{equation}
which implies that $\beta_\gamma\to 0$ as $\gamma\to+\infty$, as claimed in Theorem \ref{MainThm}. For this, it is sufficient to multiply the first equation in \eqref{EqSingGamma} by $v_1$ and to integrate by parts:
\begin{equation}\label{TestAgainstV1}
(\lambda_1-\bar{\lambda}_{\gamma}) \int_\Omega v_1 u_\gamma dy=\beta_\gamma \int_\Omega v_1 \exp(u_\gamma^2) u_\gamma dy\ge \beta_\gamma \int_\Omega v_1 u_\gamma dy\,.
\end{equation}
In view of \eqref{DefLambdaGamma}, this clearly implies \eqref{EqBetaToZero}. Now, we perform the blow-up analysis of the $u_\gamma$'s as $\gamma\to +\infty$, in order to get \eqref{WeakConv}-\eqref{QuantifWL}. Observe that we do not assume here that $(u_\gamma)_\gamma$ is bounded in $H^1_0$, as in \cite{ThiMan}. But since we are in a radially symmetric setting, we are able to start the analysis and to prove that the $u_\gamma$'s can be rescaled around $0$ in order to detect a bubble of Moser-Trudinger critical type.  Observe that our choice of $\bar{\lambda}_{\gamma}$ in \eqref{DefLambdaGamma} plays a key role for this to be true. Let $\mu_\gamma>0$ be given by
\begin{equation}\label{DefMuGamma}
\beta_\gamma \gamma^2 \exp(\gamma^2) \mu_\gamma^2=4\,,
\end{equation}
and $\tau_\gamma$ be given by
\begin{equation}\label{DefTauGamma}
u_\gamma(\mu_\gamma y)=\gamma-\frac{\tau_\gamma(y)}{\gamma}\,.
\end{equation}
Observe that, since $\Delta u_\gamma > 0$ in \eqref{EqSingGamma}, $u_\gamma$ is radially decreasing in $\Omega$, so that 
\begin{equation}\label{RadDec}
u_\gamma(r)\le u_\gamma(0)=\gamma\,,\text{ for all }r\in[0,1]\,.
\end{equation}
Here and often in the sequel, we use the identifications of \eqref{NotationRadial}.

\begin{Step}\label{St1}
We have that
\begin{equation}\label{MuTo0}
\lim_{\gamma\to +\infty} \mu_\gamma\gamma=0
\end{equation}
and
\begin{equation}\label{ConvergenceLocale}
\lim_{\gamma\to +\infty}\tau_\gamma={T_0:=}\log(1+|\cdot|^2)\text{ in }{C^{2}_{loc}}(\mathbb{R}^2)\,.
\end{equation}
\end{Step}

\begin{proof}[Proof of Step \ref{St1}]
By \eqref{DefMuGamma}, \eqref{MuTo0} is equivalent to 
\begin{equation}\label{FirstAssertion}
\lim_{\gamma\to +\infty} \beta_\gamma \exp(\gamma^2)=+\infty\,.
\end{equation}
In order to prove \eqref{FirstAssertion}, assume by contradiction that 
\begin{equation}\label{Contrad1}
\beta_\gamma \exp(\gamma^2)=O(1)\,,
\end{equation}
up to a subsequence. Then, if $w_\gamma$ is given by $u_\gamma=\gamma w_\gamma$, we have that $0\le w_\gamma \le 1$  and $w_\gamma(0)=1$. Moreover, by \eqref{DefLambdaGamma}, \eqref{EqSingGamma}, \eqref{RadDec}, \eqref{Contrad1} and standard elliptic theory, we get that
\begin{equation}\label{Contrad2}
\lim_{\gamma\to +\infty} w_\gamma=w_\infty\text{ in }C^1(\bar{\Omega})\,,
\end{equation}
where $w_\infty(0)=1$. Then, since $w_\infty>0$ in   a neighborhood of $0$, we have that 
\[
\Delta w_\gamma\ge \lambda_1 (1+o(1)) w_\infty>0
\] 
around $0$, so that there exists $\varepsilon_0\in (0,1)$ such that
\begin{equation}\label{ContrWGamma}
\varepsilon_0 |x|^2\le 1-w_\gamma(x)\,, \text{ for all } x\in \Omega\,,
\end{equation}
using that $w_\gamma$ is radially symmetric and decreasing. Thus, using \eqref{Contrad1}, \eqref{ContrWGamma} and $0\le w_\gamma \le 1$, we get that
\begin{equation}\label{Contrad3}
\begin{split}
\beta_\gamma \int_{\Omega} v_1 \exp(u_\gamma^2) u_\gamma dx~&=\beta_\gamma \exp(\gamma^2)  \gamma \int_\Omega v_1 \exp(-\gamma^2(1-w_\gamma^2)) w_\gamma dx\,,\\
&=O\left(\gamma \int_\Omega \exp(-\varepsilon_0 \gamma^2 |x|^2) dx\right)\,,\\
&=o\left(1\right)\,.
\end{split}
\end{equation}
Independently, since $w_\infty\ge 0$ and $w_\infty\not\equiv 0$, we get from \eqref{Contrad2} that
\begin{equation}\label{Contrad4}
\gamma=O\left(\int_\Omega v_1 u_\gamma dx\right)\,.
\end{equation}
Combining the equality in \eqref{TestAgainstV1}, the last estimate in \eqref{Contrad3} and \eqref{Contrad4}, we clearly get a contradition with our definition of $\varepsilon_\gamma$ below \eqref{DefLambdaGamma}. Then, \eqref{Contrad1} cannot hold true and \eqref{FirstAssertion} and \eqref{MuTo0} are proved. Finally, since \eqref{DefMuGamma}-\eqref{MuTo0} give local uniform bounds for $\Delta \tau_\gamma$,  we get \eqref{ConvergenceLocale} by now rather standard arguments (see for instance \cite[Lemma 3]{MalchMartJEMS})  and elliptic estimates. 
\end{proof}

Next, we prove that $\beta_\gamma$ does not converge to zero too fast. In the sequel it is useful to denote 
\begin{equation}\label{DefTGamma}
t_\gamma(r) := T_0\left(\frac{r}{\mu_\gamma}\right)\,= \log\left(1+\frac{r^2}{\mu_\gamma^2} \right). 
\end{equation}

\begin{Step}\label{St2}
We have  that
\begin{equation}\label{BetaNotTroSm}
|\log \beta_\gamma|=o(\gamma^2)\,.
\end{equation}
\end{Step}

\begin{proof}[Proof of  Step \ref{St2}]
In view of \eqref{EqBetaToZero}, assume by contradiction that
\begin{equation}\label{Eq1}
\lim_{\gamma \to +\infty} \frac{\log\frac{1}{\beta_\gamma}}{\gamma^2}>0\,.
\end{equation}
Here and in the sequel, we argue up to subsequences. Let $r_\gamma\in(0,1)$ be given by
\begin{equation}\label{Eq2}
\beta_\gamma \exp(u_\gamma(r_\gamma)^2)=1\,.
\end{equation}
By \eqref{EqBetaToZero}, \eqref{FirstAssertion}, and since $u_\gamma$ is radially decreasing and zero on $\partial\Omega$, $r_\gamma$ is well defined. In particular, we have that
\begin{equation}\label{Eq3}
\begin{cases}
&\beta_\gamma \exp(u_\gamma^2)>1 \text{ in }[0,r_\gamma)\,,\\
&\beta_\gamma \exp(u_\gamma^2)\le 1 \text{ in }[r_\gamma,1] \,.
\end{cases}
\end{equation}
By \eqref{Eq1} and \eqref{Eq2}, we have that
\begin{equation}\label{Eq4}
\lim_{\gamma\to +\infty} \frac{u_\gamma^2(r_\gamma)}{\gamma^2}>0\,.
\end{equation}
Observe also that $r_\gamma\gg \mu_\gamma$ by \eqref{DefTauGamma}, \eqref{ConvergenceLocale}, \eqref{FirstAssertion} and \eqref{Eq2}. Now, we prove that 
\begin{equation}\label{Eq6}
u_\gamma=\gamma-\frac{t_\gamma(1+o(1))}{\gamma}\text{ uniformly in }[0,r_\gamma]\,,
\end{equation}
as $\gamma\to +\infty$, where $t_\gamma$ is given by \eqref{DefTGamma}. For this, for any given $\eta\in(0,1)$, we let $r'_\gamma$ be given by
\begin{equation}\label{Eq5}
r'_\gamma=\sup\left\{r\in(0,r_\gamma]\text{ s.t. }\left|u_\gamma-\left(\gamma-\frac{t_\gamma}{\gamma} \right) \right|\le \eta\frac{ t_\gamma}{\gamma}\text{ in }(0,r] \right\}\,.
\end{equation}
In order to get \eqref{Eq6}, since $\eta$ may be arbitrarily small, it is sufficient to prove that 
\begin{equation}\label{CClInter}
r'_\gamma=r_\gamma\,,
\end{equation}
for all $\gamma\gg 1$. Observe that \eqref{Eq4} and \eqref{Eq5} imply that
\begin{equation}\label{RkInter}
\lim_{\gamma\to +\infty}\frac{t_\gamma(r'_\gamma)}{\gamma^2}<1\,,
\end{equation}
if $\eta$ was chosen small enough.  By the definition of $r_\gamma'$ in \eqref{Eq5}, for any $r\in [0,r_\gamma']$, we have that
\begin{equation}\label{u^2}
u_\gamma(r)^2 \le \gamma^2 + t_\gamma(r)\bra{ -2+2\eta + \frac{t_\gamma(r)}{\gamma^2}(1-2\eta + \eta^2)}.
\end{equation}
In particular,by \eqref{DefMuGamma}, \eqref{RadDec}, \eqref{RkInter} and  \eqref{u^2}, for any sufficiently small $\eta$  there exists $\kappa=\kappa(\eta)>1$ such that 
\begin{equation}\label{DomConv}
\beta_\gamma \exp (u_\gamma^2) u_\gamma\le \frac{4}{\gamma\mu_\gamma^2} \exp(-\kappa t_\gamma)\,,
\end{equation}
in $[0,r_\gamma']$, for sufficiently large $\gamma$.  Let $(s_\gamma)_\gamma$ be an arbitrary sequence such that $s_\gamma\in[0,r'_\gamma]$, for all $\gamma$. If $s_\gamma = O(\mu_\gamma)$, then, arguing as in \eqref{IPP}, we get from  \eqref{DefTauGamma}, \eqref{ConvergenceLocale}, and  \eqref{DefTGamma} that  
\begin{equation}\label{SgammaCase1}
u_\gamma'(s_\gamma) +  \frac{t_\gamma'(s_\gamma)}{\gamma} = \frac{1}{\gamma s_\gamma} \int_{0}^\frac{s_\gamma}{\mu_\gamma} ((\Delta \tau_\gamma)(s) - (\Delta T_0)(s)) s ds = o\left( \frac{s_\gamma}{\mu_\gamma^2\gamma}\right).
\end{equation}
If instead $s_\gamma\gg \mu_\gamma $, then given $R\gg 1$ we compute
\begin{equation}\label{Eq7}
\begin{split}
\int_0^{s_\gamma} (\Delta u_\gamma)(s) 2\pi s ds &=O\left(\int_0^{R\mu_\gamma} \gamma s ds\right)+O\Bigg(\int_{R\mu_\gamma}^{s_\gamma} \underset{\ge 1}{\underbrace{(\beta_\gamma \exp(u_\gamma^2))}} u_\gamma s ds\Bigg)\\
&\quad \quad+\beta_\gamma\int_0^{R\mu_\gamma}  \exp(u_\gamma^2) u_\gamma 2\pi s ds\,,\\
&=o\left(\frac{1}{\gamma} \right)+\frac{1}{\gamma}
O\Bigg(\int_{R}^{+\infty} \exp(-\kappa T_0 )r dr\Bigg)\\
 &\quad \quad +\frac{1}{\gamma}\int_0^{s_\gamma/\mu_\gamma} \frac{8\pi r}{(1+r^2)^2}  dr\,, 
\end{split}
\end{equation}
using \eqref{EqSingGamma}, \eqref{DefMuGamma}, \eqref{DefTauGamma},  \eqref{MuTo0},  \eqref{ConvergenceLocale},  \eqref{Eq3} and \eqref{DomConv}.
Then, as $R$ may be arbitrarily large in \eqref{Eq7}, we get that
\begin{equation}\label{Eq8}
\begin{split}
\int_0^{s_\gamma} (\Delta u_\gamma)(s) 2\pi s ds=o\left(\frac{1}{\gamma} \right)+\frac{1}{\gamma}\int_0^{s_\gamma/\mu_\gamma} \frac{8\pi r}{(1+r^2)^2} dr\,,  \\
= o\left(\frac{1}{\gamma} \right)-\frac{1}{\gamma}\int_0^{s_\gamma} (\Delta t_\gamma)(s) 2\pi s ds\,.
\end{split}
\end{equation}
In any case we obtain that
\begin{equation}\label{DerinSgamma}
u_\gamma'(s_\gamma) + \frac{t_\gamma'(s_\gamma)}{\gamma} =  o\left( \frac{t_\gamma'(s_\gamma)}{\gamma}\right),
\end{equation}
either by \eqref{SgammaCase1} if $s_\gamma=O(\mu_\gamma)$, or by \eqref{Eq8} arguing as in \eqref{IPP} if $s_\gamma\gg\mu_\gamma$.  Then, using that $u_\gamma(0)=\gamma$, $t_\gamma(0)=0$, and the fundamental theorem of calculus, we get \eqref{CClInter} from \eqref{DerinSgamma} and then conclude the proof of \eqref{Eq6}. Observe further that 
\begin{equation}\label{UpToRgamma}
\int_{0}^{r_\gamma} (\Delta u_\gamma)(s)2\pi s ds =\frac{4\pi+o(1)}{\gamma}\,,
\end{equation}
by \eqref{Eq8}. Now, let $(s_\gamma)_\gamma$ be such that $s_\gamma\in(r_\gamma,1]$ for all $\gamma$. We have that
\begin{equation}\label{Eq9}
\begin{split}
&\left|\int_0^{s_\gamma} (\Delta u_\gamma)(s) 2\pi s ds- \int_{0}^{r_\gamma} (\Delta u_\gamma)(s)2\pi s ds\right|\\
&\le  (\lambda_1+1) \int_{r_\gamma}^{s_\gamma}u_\gamma(s) 2\pi s ds\,,\\
&\le s_\gamma (\lambda_1+1)\sqrt{\pi}\|u_\gamma\|_2\,,
\end{split}
\end{equation}
by \eqref{EqSingGamma}, \eqref{Eq3} in $[r_\gamma,s_\gamma]$, and the  Cauchy-Schwarz inequality, where $\|\cdot\|_p$ stands for the $L^p$ norm in $\Omega$. 
Then, estimating $u_\gamma'$ as in \eqref{IPP}, we get from \eqref{UpToRgamma} and \eqref{Eq9} that
\begin{equation}\label{ExtControl}
\left|u_\gamma(r)-\frac{\log\frac{1}{r^2}}{\gamma}\right|=o\left(\frac{\log\frac{1}{r^2}}{\gamma}\right)+O\left(\|u_\gamma\|_2 \right)\,,
\end{equation}
for all $r\in[r_\gamma,1]$ and all $\gamma$, using the fundamental theorem of calculus and $u_\gamma(1)=0$.  Now, we evaluate $u_\gamma(r_\gamma)$ in both formulas \eqref{Eq6} and \eqref{ExtControl}. Since  \eqref{DefMuGamma} and \eqref{DefTGamma} imply that  
\begin{equation}\label{LowPOV}
\gamma-\frac{t_\gamma(r)}{\gamma}=\frac{1}{\gamma}\log \left(\frac{1}{\beta_\gamma (\mu_\gamma^2+ r^2)}\right)   +O\left(\frac{\log \gamma}{\gamma} \right)
\end{equation}
for any $r\in [0,1]$, and that $\log\frac{1}{r_\gamma^2}=O\left(\log\frac{1}{\mu_\gamma^2}\right) = O(\gamma^2)$, 
we get from \eqref{Eq1}, \eqref{Eq6}, \eqref{ExtControl}, and \eqref{LowPOV} that
\begin{equation}\label{CClPart1}
\lim_{\gamma\to +\infty}\frac{\|u_\gamma\|_2}{\gamma}>0\,.
\end{equation}
In order to conclude the proof, we now show that \eqref{CClPart1} contradicts our choice of $\varepsilon_\gamma$ in \eqref{DefLambdaGamma}. First, observe that $r_\gamma\to 0$, as $\gamma\to +\infty$. Otherwise, since  \eqref{Eq4} implies $\gamma = O(u_\gamma)$ in $[0,r_\gamma]$, using \eqref{Eq3} we would find that
\begin{equation}
\int_0^{r_\gamma} \beta_\gamma \exp(u_\gamma^2) u_\gamma 2\pi s ds \ge \int_0^{r_\gamma} u_\gamma 2\pi s \, ds \to +\infty
\end{equation}
as $\gamma\to +\infty$, which contradicts \eqref{UpToRgamma}.  Next we want to show that
\begin{equation}\label{Estim2}
\liminf_{\gamma\to +\infty}\inf_{r\in(r_\gamma,1]} \frac{u_\gamma(r_\gamma)-u_\gamma(r)}{\gamma(r-r_\gamma)^2}\in(0,+\infty]\,.
\end{equation}
From \eqref{EqSingGamma}, \eqref{RadDec}, \eqref{Eq3}, \eqref{CClPart1} and standard elliptic theory, we get that there exists a function $v$ such that
\begin{equation}\label{ConvLoc1}
\lim_{\gamma\to +\infty}\frac{u_\gamma}{\|u_\gamma\|_2}\to v\text{ in }C^1_{loc}(\bar{\Omega}\backslash\{0\})\,, \quad {v\ge 0}\text{ in }\Omega\backslash\{0\}\,.
\end{equation}
Moreover, by \eqref{RadDec} and \eqref{CClPart1}, we get that $u_\gamma/\|u_\gamma\|_2\to v$ strongly in $L^2$, so that $\int_{\Omega} v^2 dx=1$ and then that $v\not \equiv 0$ in $\Omega$.  Using \eqref{EqSingGamma}, \eqref{CClPart1}, \eqref{ConvLoc1} and the radial decay of the $u_\gamma$'s, there exists  $\delta\in (0,1)$ such that $\gamma \lesssim u_\gamma \lesssim \Delta u_\gamma$ in $[0,\delta]$. Then, estimating $u_\gamma'$ as in \eqref{IPP} we find 
\begin{equation*}
- u_\gamma'(r)  \gtrsim \gamma r
\end{equation*}
in $[0,1]$, and from the fundamental theorem of calculus that
 \begin{equation}\label{Estim3}
\liminf_{\gamma\to +\infty}\inf_{r\in(r_\gamma,1]} \frac{u_\gamma(r_\gamma)-u_\gamma(r)}{\gamma(r^2-r_\gamma^2)}\in(0,+\infty],
\end{equation}
which concludes the proof of \eqref{Estim2}.  Now, we compute
\begin{equation}\label{Estim5}
\begin{split}
&\int_0^1 v_1 \beta_\gamma \exp(u_\gamma^2) u_\gamma 2\pi s ds\\
&=\int_0^{r_\gamma} v_1 \beta_\gamma \exp(u_\gamma^2) u_\gamma 2\pi s ds+\int_{r_\gamma}^1 v_1 \exp(u_\gamma^2-u_\gamma(r_\gamma)^2) u_\gamma 2\pi s ds\,,\\
&=O\left(\frac{1}{\gamma} \right)+O\left(\|u_\gamma\|_\infty \int_{r_\gamma}^1 \exp(-u_\gamma(r_\gamma)(u_\gamma(r_\gamma)-u_\gamma)) s ds \right)\\
&=O\left(\frac{1}{\gamma} \right)+O\left(\gamma\int_0^{1-r_\gamma}\exp(-\varepsilon_0 \gamma^2 r^2) {(r+r_\gamma)} dr \right)= O\left(\frac{1}{\gamma} \right)+{O(r_\gamma)}\,.
\end{split}
\end{equation}
The first equality in \eqref{Estim5} uses \eqref{Eq2}; the second one uses \eqref{EqSingGamma}, \eqref{UpToRgamma}, and $u_\gamma\ge0$; in the third one, the existence of such a positive $\varepsilon_0$ is given by \eqref{Eq4} and  \eqref{Estim2}.  Independently, \eqref{CClPart1} and \eqref{ConvLoc1} with  $v\not\equiv 0$ imply that
\begin{equation}\label{Estim4}
\gamma=O\left(\int_0^1 v_1 u_\gamma r dr\right)\,.
\end{equation} 
But the equality in \eqref{TestAgainstV1}, and \eqref{Estim5}-\eqref{Estim4} clearly contradict our definition of $\varepsilon_\gamma$ in \eqref{DefLambdaGamma}. Then \eqref{Eq1} cannot be true.  This concludes the proof of \eqref{BetaNotTroSm} and that of Step \ref{St2}.
\end{proof}
Let $t_\gamma$ be as in \eqref{DefTGamma} and let $\rho_{1,\gamma},\rho_{\gamma}', \rho_{2,\gamma}>0$ be given by
\begin{equation}\label{DefRadii}
t_\gamma(\rho_{1,\gamma})=\gamma \,,\quad t_\gamma(\rho_{\gamma}')=\frac{\gamma^2}{2}\text{ and }t_\gamma(\rho_{2,\gamma})=\gamma^2-\gamma\,.
\end{equation}
 Since we have now \eqref{BetaNotTroSm}, resuming verbatim the argument in \cite[Step 3.2]{ThiMan}, there exists $\bar{C}>0$
 \begin{equation}\label{ManciniThi}
 \left|u_\gamma-\left(\gamma-\frac{t_\gamma}{\gamma}\right)\right|\le \bar{C}\frac{1+t_\gamma}{\gamma^3}\,,
 \end{equation} 
 in $[0,\rho_{\gamma}']$, and that
 \begin{equation}\label{ManciniThi2}
 \left|\beta_\gamma u_\gamma \exp(u_\gamma^2)-\frac{4\exp(-2t_\gamma)}{\gamma \mu_\gamma^2}\right|\le \frac{4\bar{C}}{3}\frac{\exp(-2t_\gamma)(1+t_\gamma^2)}{\gamma^3 \mu_\gamma^2}\,,
 \end{equation}
 in $[0,\rho_{1,\gamma}]$, for all $\gamma$. Observe that \eqref{DefRadii} and  \eqref{BetaNotTroSm} imply
\begin{equation}\label{Other}
 \rho_\gamma' \gg \rho_{1,\gamma} = \mu_\gamma  \exp ( \gamma (1/2+o(1))) \gg \mu_\gamma\,,
\end{equation} 
and that
\begin{equation}\label{Modif1}
\rho_{1,\gamma}^2 \gamma=o\left(\frac{1}{\gamma^3} \right)\,.
\end{equation}
Note that \eqref{ManciniThi2}, \eqref{Other} and \eqref{Modif1} with $u_\gamma\le \gamma$ imply  that 
 \begin{equation}\label{New2}
 \left|\int_0^{\rho_{1,\gamma}} (\Delta u_\gamma ) 2\pi r dr-\frac{4\pi}{\gamma}\right|\le  \frac{4\pi\bar{C}+o(1)}{\gamma^3}\,.
\end{equation} Let $r'_\gamma\in(0,1]$ be given by 
\begin{equation}\label{Eq2Bis}
\begin{split}
&\beta_\gamma \exp(u_\gamma(r'_\gamma)^2)=\frac{1}{\gamma^2}\text{ if }\beta_\gamma=o\left( \frac{1}{\gamma^2}\right)\,,\\
&r_\gamma'=1\text{ otherwise}\,.
\end{split}
\end{equation}
Using $u_\gamma(0)=\gamma$, $u_\gamma(1)=0$ and \eqref{BetaNotTroSm}, $r'_\gamma$ is  well defined and positive. We have in addition that
\begin{equation}\label{Eq3Bis}
\begin{cases}
&\beta_\gamma \exp(u_\gamma^2)\ge \frac{1}{\gamma^2} \text{ in }[0,r'_\gamma]\,,\\
&\beta_\gamma \exp(u_\gamma^2)\le \frac{1}{\gamma^2} \text{ in }(r'_\gamma,1] \,.
\end{cases}
\end{equation} 
Moreover, since \eqref{DefMuGamma} and \eqref{BetaNotTroSm} give
\begin{equation}\label{EstTGamma}
t_\gamma\le \gamma^2(1+o(1)) \quad \text{ in } [0,1],
\end{equation}
we get from the definition of $\rho_\gamma'$ in \eqref{DefRadii} and \eqref{ManciniThi} that $u_\gamma(\rho_\gamma')^2 = \frac{\gamma^2}{4} +O(1)$, so that  \eqref{BetaNotTroSm} and \eqref{Eq2Bis} give 
\begin{equation}\label{Rgamma'Compared}
 r_\gamma' \ge \rho_{\gamma}'\,,
\end{equation}
for $\gamma$ large enough, this inequality being obvious if $r_\gamma'=1$. Finally, we define 
\begin{equation}\label{DefRhoGamma}
\rho_\gamma = \min( \rho_{2,\gamma},r_\gamma') \,.
\end{equation}

\begin{Step}
Let $\bar C$ be as in \eqref{ManciniThi}, \eqref{ManciniThi2} and \eqref{New2}. For any $R_0>\bar{C}$ we have  
\begin{equation}\label{NewToProve}
\left|u_\gamma-\left(\gamma-\frac{t_\gamma}{\gamma} \right)\right|\le \frac{R_0}{\gamma} \quad \text{ in } [0,\rho_\gamma]\,,
\end{equation}
for sufficiently large $\gamma$.  Moreover,
\begin{equation}\label{NewToProve3}
\int_{\rho_{1,\gamma}}^{\rho_{\gamma}} (\Delta u_\gamma)(s) 2\pi s ds  = o\left(\frac{1}{\gamma^3}\right). 
\end{equation}
\end{Step}
\begin{proof}
For a given $R_0>\bar{C}$, let ${\widetilde\rho_\gamma >0}$ be given by
 \begin{equation}\label{DefRhoG}
 \widetilde \rho_\gamma =\sup\left\{r\in (0,\rho_\gamma] \text{ s.t. } \left|u_\gamma-\left(\gamma-\frac{t_\gamma}{\gamma} \right)\right|\le \frac{R_0}{\gamma}\text{ in }[0,r]\right\}\,.
 \end{equation}
Note that \eqref{ManciniThi}, \eqref{EstTGamma} and \eqref{Rgamma'Compared} imply that $\widetilde{\rho}_\gamma \ge \rho_\gamma' \ge \rho_{1,\gamma}$, for $\gamma$ large enough.  By \eqref{EqSingGamma} and \eqref{Eq3Bis}, for all $s_\gamma \in[\rho_{1,\gamma}, \widetilde \rho_\gamma]$ we compute
\begin{equation}\label{New1}
\begin{split}
\int_0^{s_\gamma} (\Delta u_\gamma) 2\pi r dr = &   \int_{0}^{\rho_{1,\gamma}}  (\Delta u_\gamma) 2\pi r dr  +  O\left( \gamma^2 \beta_\gamma \int_{\rho_{1,\gamma}}^{\widetilde \rho_\gamma} \exp(u_\gamma^2) u_\gamma   2\pi r dr \right)\,.
\end{split}
\end{equation}
Moreover, thanks to \eqref{DefMuGamma}, \eqref{EstTGamma} and \eqref{DefRhoG} we have that 
\begin{equation}\label{Observation}
\begin{split}
\int_{\rho_{1,\gamma}}^{\widetilde \rho_\gamma} \beta_\gamma \exp(u_\gamma^2) u_\gamma r  dr  & \le   \int_{\rho_{1,\gamma}}^{\rho_{{2},\gamma}} \frac{\exp(-2t_\gamma+(t_\gamma^2/\gamma^2)+O(1))}{\gamma \mu_\gamma^2} r dr\\ & =o\left(\frac{1}{\gamma^{5}} \right)\,.
\end{split}
\end{equation}  
Then, from \eqref{New2}, \eqref{New1} and \eqref{Observation} we get that 
 \begin{equation}\label{EstimLapl3}
 \left|\int_0^{s_\gamma} (\Delta u_\gamma) 2\pi r dr-\frac{4\pi}{\gamma}\right|\le \frac{4\pi\bar{C}+o(1)}{\gamma^3} \, . 
\end{equation}  
 Since $\frac{\widetilde\rho_{\gamma}}{\mu_{\gamma}} \to +\infty$ exponentially fast as  $\gamma\to +\infty$ by \eqref{Other}, we also get that 
\begin{equation}\label{New3}
\int_0^{s_\gamma} (\Delta t_\gamma )2\pi r dr = -4\pi + o\left(\frac{1}{\gamma^2}\right).  \, 
\end{equation}
Now \eqref{EstimLapl3}  and \eqref{New3} give an estimate on $u_\gamma' +  \frac{t_\gamma'}{\gamma}$ in $[\rho_{1,\gamma},\widetilde\rho_{\gamma}]$ as in \eqref{IPP} so that, by the fundamental theorem of calculus, we get 
$$
\left|\left[u_\gamma-\left(\gamma-\frac{t_\gamma}{\gamma}\right)\right]_{\rho_{1,\gamma}}^s\right|\le \frac{\bar{C}+o(1)}{\gamma^3} \log\frac{s^2}{\rho_{1,\gamma}^2}\,,
$$
so that
\begin{equation}\label{ManciniThi3}
\left|u_\gamma(s)-\left(\gamma-\frac{t_\gamma(s)}{\gamma}\right)\right|\le \left(\bar{C}+o(1)\right)\frac{1+t_\gamma(s)}{\gamma^3}\,,
\end{equation}
for all $s\in [\rho_{1,\gamma},\widetilde \rho_\gamma ]$, by \eqref{ManciniThi} at $\rho_{1,\gamma}$ and  \eqref{Other}. Then we get 
\begin{equation}\label{Equality}
\widetilde \rho_\gamma=\rho_\gamma,
\end{equation}
using $R_0>\bar{C}$ and $\widetilde \rho_\gamma > \rho_{1,\gamma}$. The proof of \eqref{NewToProve} is complete.  Finally, the estimate in \eqref{NewToProve3} follows from  \eqref{EqSingGamma},  \eqref{Eq3Bis},  \eqref{Observation} and  \eqref{Equality}. 
\end{proof}

\begin{Step} 
We have that 
\begin{equation}\label{MinorL2Norm} 
\liminf_{\gamma\to +\infty} \|u_\gamma\|_2\in(0,+\infty]\,.
\end{equation}
Moreover, we have
\begin{equation}\label{RhoTo0}
\rho_\gamma\to 0\,,
\end{equation}
and 
\begin{equation}\label{NewToProve2}
u_\gamma(r)=\frac{\log\frac{1}{r^2}}{\gamma}+o(\|u_\gamma\|_2)+v_1(r)\|u_\gamma\|_2\,  \quad \text{ in } [\rho_{\gamma} ,1]\,
\end{equation}  
as $\gamma\to +\infty$\,. 
\end{Step}

\begin{proof}
First, coming back to \eqref{TestAgainstV1}, we get that
\begin{equation*}
\begin{split}
\frac{4\pi v_1(0)+o(1)}{\gamma} =\beta_\gamma \int_0^{{\rho_{{1},\gamma}}} v_1 u_\gamma \exp(u_\gamma^2) 2\pi r dr~&\le \left(\lambda_1-{\bar{\lambda}_{\gamma}} \right) \int_0^1 v_1 u_\gamma 2\pi r dr\\
&= O\left(\frac{\|u_\gamma\|_2}{\gamma} \right)\,,
\end{split}
\end{equation*}
by \eqref{ManciniThi2}, the Cauchy-Schwarz inequality and our definition of $\varepsilon_\gamma$ below \eqref{DefLambdaGamma}. This, clearly implies \eqref{MinorL2Norm}.

 In order to prove \eqref{RhoTo0}, we observe that \eqref{DefRadii}, \eqref{DefRhoGamma}  and \eqref{NewToProve} imply 
$$
u_\gamma \ge u_\gamma (\rho_\gamma) = \gamma-\frac{t_\gamma(\rho_\gamma)}{\gamma} + o(1) \ge \gamma -\frac{t_\gamma(\rho_{2,\gamma})}{\gamma} +o(1) =1+o(1)\,,
$$
in $[0,\rho_\gamma]$ and then  that 
\[
\rho_\gamma^2  = O\bra{\int_{0}^{\rho_\gamma} u_\gamma (s) 2\pi s ds } =   O\bra{\int_{0}^{\rho_\gamma} \Delta u_\gamma (s) 2\pi s ds } = O\bra{\frac{1}{\gamma}}\,,
\] 
by \eqref{New2} and \eqref{NewToProve3}.  

Now, we turn to the proof of \eqref{NewToProve2}.  Note that 
\begin{equation}\label{useful}
\beta_\gamma \exp(u_\gamma^2) = o(1) \quad \text{ in } (\rho_\gamma,1]\,.
\end{equation}
Indeed  \eqref{useful} follows directly from \eqref{Eq3Bis} if $\rho_\gamma=r_\gamma'$, while if  $\rho_\gamma=\rho_{2,\gamma}$  then \eqref{NewToProve} gives $u_\gamma = O(1)$ in $[\rho_\gamma,1]$ and \eqref{useful} follows from \eqref{EqBetaToZero}. 
Then, by \eqref{DefLambdaGamma}, \eqref{EqSingGamma} and \eqref{useful}, we always get that
\begin{equation}\label{Observation3}
\Delta u_\gamma = (\lambda_1+o(1)) u_\gamma \quad \text{ in }[\rho_\gamma,1]\,.
\end{equation}
By elliptic estimates, this implies that \eqref{ConvLoc1} holds true with $v$ satisfying $\Delta v = \lambda_1 v $ in $\Omega\setminus \{0\}$.  We shall now prove that $v=v_1$. By \eqref{New2}, \eqref{NewToProve3} and \eqref{Observation3},  we get that
\begin{equation}\label{Observation4}
\begin{split}
\int_{0}^{s_\gamma} (\Delta u_\gamma)(r) 2\pi r dr  &  = \int_{0}^{\rho_{\gamma}} (\Delta u_\gamma)(r)  2\pi r dr   +  O(s_\gamma \|u_\gamma\|_{2}) \\ 
& =  \frac{4\pi}{\gamma} + o\left(\frac{1}{\gamma^2}\right) + O(s_\gamma \|u_\gamma\|_2)\,,
\end{split}
\end{equation}
for any sequence $s_\gamma \in [\rho_{\gamma},1]$. Estimating $u_\gamma'$ by \eqref{Observation4} in the spirit of \eqref{IPP}, using the fundamental theorem of calculus and $u_\gamma(1)=0$, and noting that $\frac{1}{\gamma^2}\log \frac{1}{\rho_\gamma^2}=O(\|u_\gamma\|_2)$, we obtain that 
\begin{equation}\label{ExtControlImp}
\left|u_\gamma-\frac{\log\frac{1}{r^2}}{\gamma}\right| = O\left(\|u_\gamma\|_2 \right)\,,
\end{equation}
for all $r\in[\rho_{\gamma},1]$. By \eqref{MinorL2Norm}, \eqref{RhoTo0} and \eqref{ExtControlImp} we get that $v$ is bounded with bounded laplacian around $\{0\}$, and then $v\in C^1(\bar{\Omega})$. Take now a sequence $(\sigma_\gamma)_\gamma$ such that $\sigma_\gamma\ge \rho_{\gamma}$, $\sigma_\gamma=o(1)$ and 
\begin{equation}\label{Inter1}
\left\|\frac{u_\gamma}{\|u_\gamma\|_2} -  v \right\|_{C^0(\bar{\Omega}\setminus B_0(\sigma_\gamma))}\to 0\,,
\end{equation} 
as $\gamma\to+\infty$. Using $u_\gamma \le \gamma$ with \eqref{Modif1} for $r\le \rho_{1,\gamma}$, \eqref{ExtControlImp} for $\rho_{1,\gamma}\le r\le \sigma_\gamma$ and \eqref{Inter1}   otherwise, we get that $(u_\gamma/\|u_\gamma\|_2)_\gamma$ converges to $v$ in $L^2$ on the whole disk, so that $\int_\Omega v^2 dx=1$, $v>0$ in $\Omega$ and $v=v_1$. Finally, we observe that for any sequence $s_\gamma\in [\rho_\gamma,1]$ we have that
\begin{equation}\label{Inter2}
\begin{split}
\int_{0}^{\text{min}(\sigma_\gamma,s_\gamma)} (\Delta u_\gamma)(r) 2\pi r dr  &= \int_{0}^{\rho_{\gamma}} (\Delta u_\gamma)(r) 2\pi r dr\, \\
&~~~~~~~~~~~~~~+ \int_{\rho_{\gamma}}^{\text{min}(\sigma_\gamma,s_\gamma)} (\Delta u_\gamma)(r) 2\pi r dr\,,\\
&= \frac{4\pi}{\gamma}+o\bra{\frac{1}{\gamma^2}} + o(\text{min}(\sigma_\gamma,s_\gamma)\|u\|_{2})\,.
\end{split}
\end{equation}
In order to get the second equality in \eqref{Inter2}, we estimate the integral up to $\rho_{\gamma}$ by \eqref{New2} and \eqref{NewToProve3}, and the integral for $\rho_{\gamma}\le r\le \min\{\sigma_\gamma,s_\gamma\}$ by  \eqref{Observation3} and \eqref{ExtControlImp}. Using \eqref{DefLambdaGamma}, \eqref{Inter1} and \eqref{Inter2},  we find that 
$$
\int_{0}^{s_\gamma}(\Delta u_\gamma)(r) 2\pi r dr  = \frac{4\pi}{\gamma} + o(s_\gamma\|u\|_2) + \lambda_1\|u\|_2\int_{0}^{s_\gamma} v_1(r)2\pi r dr,
$$
for any sequence $s_\gamma\in[\rho_\gamma,1]$. Then, \eqref{NewToProve2} follows, using  again the fundamental theorem of calculus, with \eqref{IPP} and $u_\gamma(1)=0$.
\end{proof}

Now, we conclude the proof of Theorem \ref{MainThm}. Comparing  \eqref{NewToProve} (with \eqref{LowPOV}) and \eqref{NewToProve2} at $\rho_\gamma$, we obtain that
\begin{equation}\label{ControlBeta2}
\frac{\log\frac{1}{\beta_\gamma}}{\gamma}=\|u_\gamma\|_2 v_1(0)+o\left(\|u_\gamma\|_2\right)\,,
\end{equation}
so that in particular
\begin{equation}\label{BetaSmall2}
\beta_\gamma=o\left(\frac{1}{\gamma} \right)\,,
\end{equation}
by \eqref{MinorL2Norm}. Then, arguing as in \eqref{useful} but using \eqref{BetaSmall2} instead of \eqref{EqBetaToZero} if $\rho_\gamma = \rho_{2,\gamma}$, we get that
\begin{equation}\label{usefulImproved}
\beta_\gamma \exp(u_\gamma^2) = o\left(\frac{1}{\gamma}\right)\,\quad \mbox{ in } (\rho_\gamma,1]\,.
\end{equation}
Now, on the one hand we have
\begin{equation}\label{CClClose1}
\beta_\gamma \int_{0}^1 v_1 \exp(u_\gamma^2) u_\gamma 2\pi r dr=\frac{4\pi v_1(0)+o(1)}{\gamma}+o\left(\frac{\|u_\gamma\|_2}{\gamma} \right)
\end{equation}
using \eqref{ManciniThi2} for $r\le \rho_{1,\gamma}$, \eqref{NewToProve3} for $\rho_{1,\gamma}\le r\le \rho_\gamma$, and \eqref{usefulImproved} with the Cauchy-Schwarz inequality to estimate the integral in  $[\rho_\gamma,1]$.   On the other hand, using \eqref{Modif1} with $u_\gamma\le \gamma$ for $r\in[0,\rho_{1,\gamma}]$, \eqref{NewToProve3} with $\Delta u_\gamma \ge \bar \lambda_\gamma u_\gamma $ in $[\rho_{1,\gamma},\rho_\gamma]$,  and \eqref{MinorL2Norm}-\eqref{NewToProve2} in $[\rho_{\gamma},1]$, we obtain that 
\begin{equation}\label{CClClose2}
\int_0^1 v_1 u_\gamma 2\pi r dr=\|u_\gamma\|_2(1+o(1))\,.
\end{equation}  Then, by the equality in \eqref{TestAgainstV1}, \eqref{CClClose1}, \eqref{CClClose2} and our definition of $\varepsilon_\gamma$ below \eqref{DefLambdaGamma}, we get that
\begin{equation}\label{CClClose3}
\|u_\gamma\|_2=\sqrt{\frac{l}{\lambda_1}}+o(1)\,.
\end{equation}
Now, by \eqref{EqSingGamma},  we have that 
\begin{equation}\label{Dir}
\int_\Omega |\nabla u_\gamma|^2 dx=\bar{\lambda}_{\gamma}\int_\Omega u_\gamma^2 dx+\beta_\gamma\int_\Omega u_\gamma^2 \exp(u_\gamma^2) dx,
\end{equation} 
and, using \eqref{ManciniThi}, \eqref{ManciniThi2}, \eqref{NewToProve3} and  \eqref{usefulImproved}, that 
\begin{equation}\label{final}
\begin{split}
\beta_\gamma\int_\Omega u_\gamma^2 \exp(u_\gamma^2) dx = &  \gamma(1+o(1))\beta_\gamma\int_0^{\rho_{1,\gamma}} \exp(u_\gamma^2) u_\gamma 2\pi r dr \\& +  O\left( \gamma\int_{\rho_{1,\gamma}}^{\rho_\gamma} \Delta u_\gamma 2\pi r dr \right) + o\left(\frac{\|u_\gamma\|_2}{\gamma}\right)\\
=& 4\pi +o\left(\frac{1}{\gamma}\right)
\end{split}
\end{equation}
 By \eqref{CClClose2}-\eqref{final}, since ${\bar{\lambda}_{\gamma}}=\lambda_1 +o(1)$ by \eqref{DefLambdaGamma}, we prove that \eqref{QuantifWL} holds true. Moreover, by \eqref{QuantifWL} with \eqref{RhoTo0} and \eqref{NewToProve2},  there must be the case that \eqref{WeakConv} holds true. Thus, Theorem \ref{MainThm} is proved.
\end{proof}

\section{Proof of Theorem \ref{MainThmNod}}\label{SectPfTheorem1GeneralCase}
Let $\Omega\subset \mathbb{R}^2$ be the unit disk centered at $0$. Let $\lambda_1,\lambda_2, \ldots $ and $v_1,v_2,\ldots$ be as above Theorem \ref{MainThm}. By Bessel functions' theory, $v_1$ extends in a unique way to a radial function $\bar{v}_1$, satisfying $\Delta \bar{v}_1=\lambda_1 \bar{v}_1$ in $\mathbb{R}^2$. It is known that $\bar{v}_1$ vanishes exactly for $|x|=r_n:=\sqrt{\lambda_n/\lambda_1}$ with $\bar{v}_1'(r_n)\neq 0$ (see \eqref{NotationRadial}) for any $n\ge 1$. Moreover, we have that $\bar{v}_n:=\bar{v}_1(r_n\cdot)$ is proportional to $v_n$ in $\Omega$, namely
\begin{equation}\label{DefAlphak}
v_n = \frac{r_n}{\sqrt{\alpha_n}} \bar v_1(r_n\cdot) \quad \text{ where } \quad \alpha_n = \int_{B_0(r_n)} \bar v_1^2 dx\,,
\end{equation}
for all integer $n\ge 1$. Let $l>0$ be a fixed real number and $k\ge 2$ be a fixed integer. Let us define 
\begin{equation}\label{DefLambdaGammak}
\tilde \lambda_{\gamma}:=\lambda_1-\frac{4\pi v_1(0)}{\gamma}\sqrt{\frac{\lambda_1}{\tilde l}} \quad \text{with} \quad \tilde l = \frac{l}{ \alpha_k}\,.
\end{equation}
By Theorem \ref{MainThm} (with $l=\tilde l$) for all $\gamma$ large enough, there exist $\tilde \beta_{\gamma}>0$ and a smooth radial function $\tilde u_{\gamma}$ such that
\begin{equation}
\begin{cases}
\Delta \tilde u_{\gamma} =\tilde  \lambda_{\gamma} \tilde u_{\gamma} +\tilde \beta_{\gamma}  \tilde u_{\gamma} \exp(\tilde u_{\gamma}^2),  \tilde u_{\gamma}>0 \quad \text{ in }\Omega\,,\\
\tilde u_{\gamma} = 0 \text{ on }\partial \Omega\,,\\
\tilde u_\gamma(0)=\gamma\,,
\end{cases}
\end{equation}
and we have  
\begin{equation}\label{C1ConvLastPart}
\tilde \beta_{\gamma}=o(1)\text{ and } \tilde u_{\gamma} \to \tilde u_\infty:=v_1 \sqrt{\frac{\tilde l}{\lambda_1}}\text{ in }C^1_{loc}(\bar{\Omega}\backslash\{0\})\cap L^2(\Omega)\,,
\end{equation}
 as $\gamma\to +\infty$. Moreover, as discussed in Lemma \ref{ExistRes}, we have that $\tilde u_{\gamma}$ is globally defined on $\R^2$ and, as a consequence of \eqref{C1ConvLastPart} and standard ODE theory, we find that 
\begin{equation}\label{ConvExtended}
\tilde u_{\gamma} \to \bar{v}_1 \sqrt{{\tilde l}/{\lambda_1}} \quad \text{ in } \quad C^1_{loc}(\mathbb{R}^2\backslash\{0\})\cap L^2_{loc}(\R^2).
\end{equation} 
Besides, by the implicit function theorem and ODE theory, there exists $\varepsilon_0>0$ and $\bar{\gamma}\gg 1$ large such that $\tilde u_{\gamma}$ vanishes exactly once in $I_k:=(r_k-\varepsilon_0, r_k+\varepsilon_0)$ at some $r_{k,\gamma}$,  for all  $\gamma\ge \bar{\gamma}$ (indeed $\tilde u_\gamma$ vanishes exactly $k-$times in $[0,r_k+\eps_0)$ for $\gamma\gg1$). By construction, we also have that $r_{k,\gamma}=r_k+o(1)$ as $\gamma\to+\infty$. Setting $u_{\gamma}:=\tilde u_{\gamma}(r_{k,\gamma}\cdot)$, we get from the above discussion that $u_{\gamma}$ solves \eqref{EqSingGammaNod} with $\beta_{\gamma}=r_{k,\gamma}^2 \tilde \beta_{\gamma}$ and $\bar{\lambda}_{\gamma}=r_{k,\gamma}^2 \tilde \lambda_{\gamma}=\lambda_k+o(1)$. Besides  \eqref{ConvExtended} implies
 \begin{equation}\label{LocConvFurtherLastPart}
u_{\gamma} \to  \bar{v}_1(r_k\cdot) \sqrt{\frac{\tilde l}{\lambda_1}} = v_k  \sqrt{\frac{l}{\lambda_k}} \ \text{ in }C^1_{loc}(\bar{\Omega}\backslash\{0\})\cap L^2(\Omega)\,,
 \end{equation}
 as $\gamma\to +\infty$. Then, using the invariance of the $L^2$-norm of the gradient under dilation in dimension 2, \eqref{QuantifWL} for $\tilde u_{\gamma}$ if $|x|\le 1/r_{k,\gamma}$, and \eqref{LocConvFurtherLastPart} if $|x|\ge 1/r_{k,\gamma}$, we get that
 \begin{equation}\label{QuantificationThm2}
 \begin{split}
 \int_\Omega |\nabla u_{\gamma}|^2 dx &=\int_\Omega |\nabla \tilde u_{\gamma} |^2 dy+\int_{\Omega\backslash B_0(1/r_{k,\gamma})} |\nabla u_{\gamma}|^2 dx\,,\\
 &=4\pi +\tilde l+\tilde l \int_{B_0(r_k)\backslash \Omega} \bar{v}_1^2 dy+o(1)\,,\\
 &=4\pi +l+o(1)\,,\\
 \end{split}
 \end{equation}
so that \eqref{QuantifWL} holds true.  Clearly  \eqref{LocConvFurtherLastPart}  and \eqref{QuantificationThm2} give also \eqref{WeakConv}. Finally, we shall prove that  \eqref{relEpsGamma} holds. In order to do this, we may multiply \eqref{EqSingGammaNod}  by $v_k$, integrate by parts and use \eqref{LocConvFurtherLastPart} to get
\begin{equation}\label{ConclThm2}
\begin{split}
 \left(\lambda_k-\bar{\lambda}_{\gamma}\right)(1+o(1))\sqrt{\frac{l}{\lambda_k}} & = \left(\lambda_k-\bar{\lambda}_{\gamma}\right)\int_\Omega u_\gamma v_k dy \\
 &=\beta_{\gamma} \int_\Omega u_{\gamma} \exp(u_{\gamma}^2) v_k dy\,.  
\end{split}
\end{equation}
Arguing as in \eqref{usefulImproved}-\eqref{CClClose2} for the $\tilde u_{\gamma}'s$, we find that 
\begin{equation}\label{ConclThm2-part2}
\begin{split}
\beta_{\gamma}\int_{B_0(\frac{1}{r_{k,\gamma}})} u_{\gamma} \exp(u_{\gamma}^2) v_k dy &=\tilde \beta_{\gamma} \int_\Omega \tilde u_{\gamma} \exp(\tilde u_{\gamma}^2) v_k dy   \\ 
&=\frac{4\pi v_k(0)(1+o(1))}{\gamma}\,,
\end{split}
\end{equation}
and that $\tilde \beta_{\gamma}=o(\frac{1}{\gamma})$, so that  \eqref{LocConvFurtherLastPart} implies
\begin{equation}\label{ConclThm2-part3}
\beta_{\gamma}\int_{\Omega\setminus B_0(\frac{1}{r_{k,\gamma}})} u_{\gamma} \exp(u_{\gamma}^2) v_k dy  = o\bra{\frac{1}{\gamma}}. 
\end{equation}
Using \eqref{ConclThm2}-\eqref{ConclThm2-part3} we conclude the proof of \eqref{relEpsGamma}, and  Theorem \ref{MainThmNod} is proved.

\section{Proof of Theorem \ref{Thm3}}
As in Section \ref{SectPfThm1PositiveCase}, we first prove the existence part of Theorem \ref{Thm3}. The argument is similar to the one of \cite[Proof of Theorem 1]{MalchMartJEMS}.

\begin{lem}\label{LemmaExistNew}
Let $g$ be as in Theorem \ref{Thm3}. For any $\gamma >0$ there exist a unique real number $\beta_\gamma>0$ and a unique smooth radial function $u_\gamma$ in $\Omega$  solving 
\begin{equation}\label{EqLemmaThm3}
\begin{cases}
&\Delta u= \beta_\gamma u g(u)\,,~~ u>0\text{ in }\Omega\,,\\
&u=0\text{ on }\partial\Omega\,,\\
& u(0)=\gamma\,.
\end{cases}
\end{equation}
Moreover, we have
\begin{equation}\label{BoundBetaGamma}
0<\beta_\gamma<\lambda_1 \bra{\inf_{[0,+\infty]} g}^{-1}.
\end{equation}
\end{lem}
\begin{proof}
Let $f:\R\rightarrow \R$ be the continuous odd extension of the function $t\mapsto t g(t)$, $t\ge 0$. For any $\gamma >0$, let $w_\gamma$ be the solution the ODE
\begin{equation}
\begin{cases}\label{Prob}
&\Delta w_\gamma = f(w_\gamma)\,,\\
&w_\gamma(0)=\gamma\,.
\end{cases}
\end{equation}
Arguing as in the proof of Lemma \ref{ExistRes}, we get that $w_\gamma$ is defined on $\R^2$, since the function $E(w_\gamma)= \frac{1}{2}|w_\gamma'|^2 + F(w_\gamma)$ with $F(t)=\int_0^{t} f(s)ds$ is nonincreasing in the existence interval for $w_\gamma$, and since $F(t)\to+\infty$ as $t\to +\infty$. Let now \begin{equation}
R_\gamma=\sup\left\{r>0\,\text{ s.t. }w_\gamma>0\text{ in }[0,r] \right\}\,.
\end{equation} 
Clearly  $R_\gamma>0$ and for any fixed $\eps_\gamma\in (0,R_\gamma)$ and and $r\in [\eps_\gamma,R_\gamma)$ have  
\begin{equation}
- r w_\gamma'(r)=\int_0^r (\Delta w_\gamma)(s) s ds\ge \int_0^{\eps_\gamma}f(w_\gamma(s))s ds>0\,.
\end{equation}
Using the fundamental theorem of calculus as in \eqref{IntermComput} we get that $R_\gamma<+\infty$, so that  $R_\gamma$ is the first zero of $w_\gamma$. Then, the function $u_\gamma= w_\gamma(R_\gamma\cdot )$ satisfies \eqref{EqLemmaThm3} with $\beta_\gamma = R_\gamma^2$. Multiplying the equation in \eqref{EqLemmaThm3} by $v_1$ and integrating by parts, we get that
$$
\lambda_1 \int_{\Omega} v_1 u_\gamma dx = \beta_\gamma \int_{\Omega}v_1 u_\gamma g(u_\gamma) dx \ge \beta_\gamma {\left(\inf_{[0,+\infty]}g \right)} \int_{\Omega} v_1 u_\gamma dx\,,
$$ which gives \eqref{BoundBetaGamma}. Note that {\eqref{G1} and \eqref{G2} imply} $\inf_{[0,+\infty]}g>0$. Finally, we observe that $\beta_\gamma$ and $u_\gamma$ are uniquely determined by $\gamma$ and $g$. Indeed if $\bar \beta>0$ and $\bar u\in C^1(\bar \Omega)$ are such that
$$
\begin{cases}
\Delta \bar u = \bar \beta \bar u g(\bar u)\,, \quad  \bar u >0,  \text{ in  }\Omega\\
\bar u = 0 \quad \text{ on }\partial \Omega,\\
\bar u(0)=\gamma,\\
\bar u \text{ is radially symmetric in } \text{ in  }\Omega\,,
\end{cases}
$$ 
then uniqueness theory for solution of ODEs implies  $\bar u(\frac{\cdot}{ \sqrt{\bar\beta}}) = w_\gamma$. In particular $\bar \beta= \beta_\gamma = R_\gamma^2$ and $\bar u = u_\gamma$. 
\end{proof}

For any $\gamma>0$ let $\beta_\gamma$ and $u_\gamma$ be as in Lemma \ref{LemmaExistNew}. As in Section \ref{SectPfThm1PositiveCase} we shall rescale around 0 and prove that the $u_\gamma$'s are close to a Moser-Trudinger bubble up to a sufficiently large scale. However, here more precise expansions (see \eqref{UpToRhoGamma} and \eqref{SecondExpansion}) are needed in order to detect the effect of the term $e^{-a u}$ on the shape of such bubble. 

Let us define $\mu_\gamma>0$ such that 
\begin{equation}\label{DefMuG}
\mu_\gamma^2 \beta_\gamma \gamma^2  g(\gamma)=4.
\end{equation} Note that 
\begin{equation}\label{MuTo0G}
\mu_\gamma\to 0 \quad \text{ as } \quad \gamma\to +\infty.
\end{equation} 
Otherwise, by \eqref{DefMuG}, we could find a subsequence such that $\beta_\gamma\gamma^2g(\gamma)=O(1)$. Since $u_\gamma \le u_\gamma(0)= \gamma$ in $\Omega$ and since $g(\gamma)\to+\infty$ as $\gamma\to +\infty$, the assumptions \eqref{G1} and \eqref{G2} imply that $g(u_\gamma)\le g(\gamma)$ for large values of $\gamma$.  Hence,  we would have  
$$
\Delta u_\gamma \le  \beta_\gamma \gamma g(\gamma)= O(\gamma^{-1}),
$$
which contradicts $u_\gamma (0)=\gamma \to +\infty$ by standard elliptic estimates.
Now let $\tau_\gamma$ be defined as in \eqref{DefTauGamma} with $\mu_\gamma$ as in \eqref{DefMuG}. Then, we will show that 
$$
\tau_\gamma \to T_0 \qquad \text{ and } \qquad \gamma(\tau_\gamma -T_0)\to  a S_0 \quad \text{ in }C^{2}_{loc}(\R^{2}),
$$
where $T_0$ is as in \eqref{ConvergenceLocale} and 
\begin{equation}\label{DefS0}
S_0(x) = -\frac{1}{2} T_0(x) + \frac{1}{2} \frac{|x|^2}{1+|x|^2}\,.
\end{equation}
Note that  we have
\begin{equation}\label{EqT0}
\Delta T_0  + 4\exp(-2T_0)=0  \quad \text{ in }\R^2\,,
\end{equation}
and  \begin{equation}\label{EqS0}
\Delta S_0 -8 \exp(-2T_0) S_0 = 4T_0 \exp(-2T_0)  \quad \text{ in }\R^2\,.
\end{equation}
More precisely, setting 
\begin{equation}\label{DefTSGamma}
t_\gamma  = T_0 \left( \frac{\cdot}{\mu_\gamma }\right) \qquad \text{ and }\qquad S_\gamma  = S_0 \left( \frac{\cdot}{\mu_\gamma }\right),
\end{equation}
and letting $\rho_\gamma>0$ be defined by
\begin{equation}\label{DefRhoGammaG}
t_\gamma (\rho_\gamma)=\frac{\gamma^2}{2},
\end{equation}
we get the following expansion in $[0,\rho_\gamma]$. 

\begin{Step}\label{ExpUpToRhoGamma}  As $\gamma\to+\infty$, we have that 
\begin{equation}\label{UpToRhoGamma}
u_\gamma = \gamma -\frac{t_\gamma}{\gamma} + \frac{a S_\gamma}{\gamma^2}  + O\left(\frac{t_\gamma}{\gamma^3} \right)\qquad \text{ in }[0,\rho_\gamma]\,,
\end{equation}
with $T_\gamma$ and $S_\gamma$ as in \eqref{DefTSGamma}. Moreover, we have that 
\begin{equation}\label{LapUpToRhoGamma}
\int_{B_0(\rho_\gamma)} \Delta u_\gamma dy = \frac{4\pi}{\gamma}+\frac{2\pi a }{\gamma^2} + O\left(\frac{1}{\gamma^3}\right),
\end{equation}
and that 
\begin{equation}\label{ExpEnergy1}
\int_{B_0(\rho_\gamma)} u_\gamma \Delta u_\gamma dy = 4\pi +o(1)\,.
\end{equation}
\end{Step}
\begin{proof}
The proof of Step \ref{ExpUpToRhoGamma} is similar to the one of \cite[Step 3.2]{ThiMan}.  Observe that \eqref{G1}, \eqref{BoundBetaGamma} and \eqref{DefMuG} imply that $\mu_\gamma^{-2} = O(\gamma^2 \exp(\gamma^2))$. Then, we get 
\begin{equation}\label{BoundTGamma}
t_\gamma(r)\le \log\bra{1+\frac{1}{\mu_\gamma^2}} =  O(\gamma^2) \, , 
\end{equation}
for any $r\in [0,1]$. This will be used several times in the sequel. 

 Let $w_\gamma$ be defined by 
\begin{equation}\label{DefWG}
u_\gamma = \gamma -\frac{t_\gamma}{\gamma}+ \frac{a w_\gamma}{\gamma^2},
\end{equation} 
and let
\begin{equation}\label{DefRhoGamma'}
\rho_\gamma' = \sup\{r \in [0,\rho_{\gamma}] \;:\; \left|w_\gamma-S_\gamma \right|\le 1+t_\gamma  \}. 
\end{equation}
First, by   \eqref{DefRhoGammaG} and \eqref{BoundTGamma}-\eqref{DefRhoGamma'}, one has that  $u_\gamma \ge \frac{\gamma}{2} + O(1)$  in $[0,\rho_\gamma']$, so that \eqref{G1} gives $g(u_\gamma) = \exp(u_\gamma^2-a u_\gamma)$ for $\gamma$ sufficiently large.  Moreover, \eqref{BoundTGamma}, \eqref{DefWG} and \eqref{DefRhoGamma'} imply that 
\begin{equation}\label{Exponent}
\begin{split}
u_\gamma^2 - a u_\gamma 
&= \gamma^2- a\gamma  -2t_\gamma  + \psi_\gamma + O\bra{\frac{1+t_\gamma}{\gamma^2}},
\end{split}
\end{equation}
in $[0,\rho_\gamma']$, where
$$
\psi_\gamma = a \frac{2w_\gamma+t_\gamma}{\gamma}+\frac{t_\gamma^2}{\gamma^2} -\frac{2a t_\gamma w_\gamma}{\gamma^3}\,.
$$ 
Using the simple inequality 
$|e^{x}-1-x|\le e^{|x|} |x|^2$ for $x\in \mathbb{R}$, we get that 
\begin{equation}\label{Error}
\begin{split}
\exp\bra{\psi_\gamma}& =  1 + \psi_\gamma +  O\bra{\exp\bra{|\psi_\gamma|}\psi_\gamma^2}\\
&=  1+ a\frac{2w_\gamma+t_\gamma}{\gamma} + O\bra{\frac{1+t_\gamma^2}{\gamma^2}} + O\bra{\exp(|\psi_\gamma|) \frac{1+t_\gamma^4}{\gamma^2}} \\
&=  1+ a\frac{2w_\gamma+t_\gamma}{\gamma}  + O\bra{\exp(|\psi_\gamma|) \frac{1+t_\gamma^4}{\gamma^2}}\,,
\end{split}
\end{equation}
in $[0,\rho_\gamma']$. Since $u_\gamma =  \gamma \left( 1+O\left(\frac{1+{t_\gamma}}{\gamma^2}\right)   \right)$ by \eqref{BoundTGamma}, \eqref{DefWG} and \eqref{DefRhoGamma'}, using \eqref{EqLemmaThm3}, \eqref{BoundTGamma}, \eqref{Exponent}, \eqref{Error} and the definition of $\mu_\gamma$ in \eqref{DefMuG}, we obtain that 
\begin{equation}\label{ExpLap}\begin{split}
\Delta u_\gamma & =  {\beta_\gamma \gamma g(\gamma)} \exp(-2t_\gamma+\psi_\gamma) \left( 1+O\left(\frac{1+{t_\gamma}}{\gamma}\right)   \right) \\  
&=\frac{4}{\gamma\mu_\gamma^2} \exp\bra{-2t_\gamma} \bra{ 1 + a\frac{2w_\gamma+t_\gamma}{\gamma} + O\bra{\exp(|\psi_\gamma|) \frac{1+t_\gamma^4}{\gamma^2}} }\,,
\end{split}
\end{equation} 
in $[0,\rho_\gamma']$. In particular,   \eqref{EqT0},  \eqref{EqS0}, \eqref{DefWG} and \eqref{ExpLap} yield
\begin{equation}\label{EqDiff}
\Delta (w_\gamma - S_\gamma) =\frac{4}{ \mu_\gamma^2 }\exp(-2t_\gamma)\bra{ 2(w_\gamma-S_\gamma) + O\bra{\exp(|\psi_\gamma|) \frac{1+t_\gamma^4}{\gamma}}}.
\end{equation}
Note that, since $t_\gamma\le \frac{\gamma^2}{2}$ in $[0,\rho_\gamma']$ by \eqref{DefRhoGammaG}, we get
\[
\begin{split}
-2t_\gamma+|\psi_\gamma |  &= t_\gamma \left(-2  + \frac{t_\gamma}{\gamma^2} + O\left(\frac{1}{\gamma}\right)\right) + o(1)\,,\\
 & \le  t_\gamma \left(-\frac{3}{2} + O\left(\frac{1}{\gamma}\right)\right) +o(1)\,.
\end{split}
\]
Hence, there exists $\kappa>1$ such that 
\begin{equation}\label{kappa}
(1+t_\gamma^4)\exp(-2t_\gamma+|\psi_\gamma|) = O(\exp(-\kappa t_\gamma))\quad \text{ in }[0,\rho_\gamma']\,,
\end{equation}
for sufficiently large $\gamma$. By \eqref{EqDiff} and \eqref{kappa}, we can find $C_1>0$ such that 
\[\begin{split}
r|(w_\gamma-S_\gamma)'(r)| &\le  \int_{0}^r   |\Delta(w_\gamma - S_\gamma)| s ds \,, \\
& \le  \frac{8  \|w_\gamma'-S_\gamma'\|_{L^\infty([0,\rho_\gamma'])  } }{ \mu_\gamma^2 } \int_{0}^{r} \exp(-2t_\gamma)  s^2  ds  \\ 
& \quad \quad + \frac{1}{\mu^2_\gamma\gamma }\int_{0}^r O(\exp(-\kappa t_\gamma (s)))s\, ds \,, \\
& \le   C_1  r \frac{\frac{r^2}{\mu^2_\gamma}}{1+\frac{r^3}{\mu_\gamma^3}} \|w_\gamma'-S_\gamma'\|_{L^\infty([0,\rho_\gamma'])}   + \frac{C_1}{\gamma} \frac{\frac{r^2}{\mu_\gamma^2}}{1+\frac{r^2}{\mu_\gamma^2}},
\end{split}
\]
for any $r\in [0,\rho_\gamma']$.  Therefore we have that
\begin{equation}\label{EstDer}
|(w_\gamma-S_\gamma)'(r)|\le \frac{C_2 \frac{r}{\mu_\gamma}}{1+ \frac{r^2}{\mu_\gamma^2}}\left(  \|w_\gamma'-S_\gamma'\|_{L^\infty([0,\rho_\gamma'])} + \frac{1}{\mu_\gamma \gamma} \right),
\end{equation}
for some constant $C_2>0$. We claim now that
\begin{equation}\label{ToProve}
\begin{split}
\|(w_\gamma-S_\gamma)'\|_{L^\infty([0,\rho_\gamma'])} =O\left(\frac{1}{\mu_\gamma \gamma}\right). 
\end{split}
\end{equation}
Otherwise there exists $0<\rho_\gamma''\le \rho_\gamma'$ such that 
\begin{equation}\label{Contr}
\begin{split}
\mu_\gamma \gamma \|(w_\gamma-S_\gamma)'\|_{L^\infty([0,\rho_\gamma'])} = \mu_\gamma \gamma  |(w_\gamma- S_\gamma)'(\rho_\gamma'')|\to +\infty. 
\end{split}
\end{equation}
Note that \eqref{EstDer} and \eqref{Contr} imply $\rho_{\gamma}''=O(\mu_\gamma)$  and $\mu_\gamma = O(\rho_\gamma'')=O(\rho_\gamma')$, so that, up to a subsequence we get 
\begin{equation}\label{Compare}
\frac{\rho_\gamma''}{\mu_{\gamma}}\to \delta\in (0,+\infty) \quad\text{ and }\quad \frac{\rho_{\gamma}'}{\mu_\gamma} \to \delta_0\in [\delta,+\infty]\,.
\end{equation} 
We define 
$$
z_\gamma = \frac{(w_\gamma -S_\gamma)(\mu_\gamma \cdot)}{\mu_\gamma \|(w_\gamma-S_\gamma)'\|_{L^\infty([0,\rho_\gamma'])} }\,,
$$ 
so that we have $|z_\gamma'|\le 1$ in $[0,\frac{\rho_\gamma'}{\mu_\gamma}]$. Using \eqref{Compare}, $z_\gamma(0)=0$ the fundamental theorem of calculus, and then the ODE \eqref{EqDiff} and \eqref{Contr}, we observe that 
$$\left(\|z_\gamma\|_{L^\infty\left(\left[0,{\rho_\gamma''}/{{\mu_\gamma}}\right]\right)}\right)_\gamma \text{ and }\left(\|\Delta z_\gamma\|_{L^\infty\left(\left[0,{\rho_\gamma''}/{{\mu_\gamma}}\right]\right)}\right)_\gamma$$
 are both bounded sequences.  Then by radial elliptic estimates, up to a subsequence we do not only get that $z_\gamma \to z_0$ in $C^{1}_{loc}({B_0(\delta_0)})$, but we also get that $z_\gamma'(\rho_\gamma'')\to z_0'(\delta)$ as $\gamma\to +\infty$, where $z_0\in C^1(\overline{B_0(\delta)})$ satisfies
$$
\begin{cases}
\Delta z_0 -8 \exp(-2T_0)z_0=0\text{ in } B_0(\delta_0)\cap \overline{B_0(\delta)} \,, \\
z_0(0)=0\,.
\end{cases}
$$ 
But, since $z_0$ is radially symmetric, we get $z_0\equiv 0$ and then $z_\gamma'(\rho_\gamma'')\to 0$. But this contradicts our definition of $\rho_\gamma''$  in \eqref{Contr}, which gives that $|z_\gamma'(\rho_\gamma'')|=1$. Hence, \eqref{Contr} cannot hold  and  \eqref{ToProve} is proved.  Then \eqref{EstDer} reads as 
$$
|(w_\gamma-S_\gamma)'|\le \frac{2C_{3}}{\mu_\gamma\gamma} \frac{\frac{r}{\mu_\gamma}}{1+ \frac{r^2}{\mu_\gamma^2}} = \frac{C_{3}}{\gamma} t_\gamma'(r)\,,
$$
in  $[0,\rho_\gamma']$. In particular, using the fundamental theorem of calculus, we get  $w_\gamma - S_\gamma = O(\frac{t_\gamma}{\gamma})$  in  $[0,\rho_\gamma']$, so that  $\rho_\gamma'=\rho_\gamma$ for large $\gamma$ and \eqref{UpToRhoGamma} holds.  Finally, \eqref{LapUpToRhoGamma} and \eqref{ExpEnergy1} follow from \eqref{ExpLap}, \eqref{kappa} and 
\begin{equation}\label{Scales}
\rho_\gamma^2 = \mu_\gamma^2 \left(\exp\left(\frac{\gamma^2}{2}\right) -1\right). 
\end{equation}   
\end{proof}

Note that the expansion  in \eqref{UpToRhoGamma}, \eqref{DefS0}, and \eqref{Scales} imply that   
\begin{equation}\label{AtRhoGamma}
u_\gamma(\rho_\gamma)= \gamma -\frac{t_\gamma(\rho_\gamma)}{\gamma}\bra{1+\frac{a}{2\gamma}} + o\bra{\frac{t_\gamma(\rho_\gamma)}{\gamma^2}},
\end{equation}
as $\gamma\to +\infty$. Let us now fix any $c_1>\max\{c_0,\frac{a}{2}\}$, where $c_0$ is as in \eqref{G1}.  Let  $r_\gamma>0$ be such that  
\begin{equation}\label{DefRGamma}
u_\gamma(r_\gamma)= c_1.
\end{equation}
We shall prove that an expansion similar to \eqref{AtRhoGamma}  holds uniformly in $[\rho_\gamma,r_\gamma]$.

\begin{Step}\label{LemmaImprovedExpansion}
As $\gamma\to +\infty$, we have that
\begin{equation}\label{SecondExpansion}
u_\gamma = \gamma -\frac{t_\gamma}{\gamma}\bra{1+\frac{a}{2\gamma}} + o\bra{\frac{t_\gamma}{\gamma^2}}
\end{equation}
uniformly in $[\rho_\gamma,r_\gamma]$.  Moreover, there exists $\delta>0$ such that
\begin{equation}\label{LapInter}
\int_{B_0(r_\gamma)\setminus B_0(\rho_\gamma)} \Delta u_\gamma dy = O(\exp(-\delta \gamma)). 
\end{equation}
\end{Step}
\begin{proof} Let us denote $\alpha_\gamma = \bra{1+\frac{a}{2\gamma}}$ and  $ \mathcal T_\gamma  = \alpha_\gamma t_\gamma$. Arguing as in \eqref{BoundTGamma}, we get that 
\begin{equation}\label{BoundTauGamma}
\mathcal T_\gamma=O(\gamma^2)
\end{equation}
in $[0,1]$. For any  fixed $\eta \in(0,1)$ we set 
$$
r_{\gamma}'  = \sup\left\{ r \in [\rho_\gamma,r_\gamma] \;:\;  \left| u_\gamma -\bra{ \gamma -\frac{\mathcal{T}_\gamma}{\gamma}}\right|\le \frac{\eta \mathcal{T}_\gamma}{\gamma^2}   \right \}. 
$$
Note that $r_\gamma'>\rho_\gamma$ by \eqref{AtRhoGamma}. We claim that there exists $\delta>0$ such that 
\begin{equation}\label{Claim}
\int_{\rho_\gamma}^{r_\gamma'} \Delta u_\gamma 2\pi r \, dr = O\left(\exp\left(-\frac{\delta}{2} \gamma\right)\right)\,. 
\end{equation}
As in the proof of Step \ref{ExpUpToRhoGamma}, we write 
\begin{equation}\label{DefPhiGamma}
u_\gamma = \gamma  - \frac{\mathcal{T}_\gamma}{\gamma} + \frac{\ph_\gamma}{\gamma^2} \quad \text{ with }|\ph_\gamma|\le \eta\mathcal{T}_\gamma \quad \text{ in } [\rho_\gamma,r_\gamma']\,.
\end{equation}
In particular, using \eqref{BoundTauGamma} and \eqref{DefPhiGamma}, we can write   
$$
u_\gamma^2-a u_\gamma \le  \gamma^2 -a \gamma - 2\mathcal{T}_\gamma +\frac{a\mathcal{T}_\gamma}{\gamma}+ \frac{\mathcal{T}_\gamma^2}{\gamma^2}  + C_1 \frac{\eta \mathcal{T}_\gamma }{\gamma}+C_2\,,
$$
where $C_1$,$C_2>0$ do not depend on the choice of $\eta$. Since $c_1>c_0$  and $u_\gamma$ is radially decreasing,  \eqref{G1}, \eqref{EqLemmaThm3} and \eqref{DefMuG}  imply that 
\[\begin{split}
\Delta u_\gamma & =  \beta_\gamma  u_\gamma \exp(u_\gamma^2-a u_\gamma)\,, \\ 
&\lesssim \beta_\gamma \gamma \exp(\gamma^2-a \gamma) \exp\bra{-2\mathcal{T}_\gamma +\frac{a\mathcal{T}_\gamma}{\gamma}+ \frac{\mathcal{T}_\gamma^2}{\gamma^2}+ C_1 \eta \frac{\mathcal{T}_\gamma}{\gamma} }\,, \\
&\lesssim  \frac{4}{\mu_\gamma^2 \gamma}  \exp\bra{-2\mathcal{T}_\gamma +\frac{a\mathcal{T}_\gamma}{\gamma}+ \frac{\mathcal{T}_\gamma^2}{\gamma^2}+ C_1 \eta \frac{\mathcal{T}_\gamma}{\gamma}}\,.
\end{split}
\]
Integrating in the interval  $[\rho_\gamma,r_\gamma']$  and using the change of variable $\tau = \mathcal{T}_\gamma(r) $ so that 
$$  
\frac{ r \, dr}{\mu_\gamma^2} = \frac{\exp(\frac{\tau}{\alpha_\gamma})d\tau }{2 \alpha_\gamma} \,,
$$
we get that
\[\begin{split}
\int_{\rho_\gamma}^{r_\gamma'} \Delta u_\gamma 2\pi r \, dr & \lesssim\int_{\tau(\rho_\gamma)}^{\tau(r_\gamma')} \frac{1}{\gamma}  \exp\bra{-2\tau +\frac{a\tau}{\gamma}+ \frac{\tau^2}{\gamma^2}  + C_1\frac{\eta\tau }{\gamma}  + \frac{\tau}{\alpha_\gamma} }  d\tau\,. 
\end{split}
\] 
Since $\frac{1}{\alpha_\gamma} = 1-\frac{a}{2\gamma} +O(\frac{1}{\gamma^2})$ and $\frac{\tau}{\gamma^2} \le \frac{\mathcal T(r_\gamma')}{\gamma^2}=O(1)$ by \eqref{BoundTauGamma}, we find that 
\begin{equation}\label{Integral}
\begin{split}
\int_{\rho_\gamma}^{r_\gamma'} \Delta u_\gamma 2\pi r \, dr 
&\lesssim \frac{1}{\gamma}  \int_{\tau(\rho_\gamma)}^{\tau(r_\gamma')} \exp\bra{ -\tau \bra{ 1- \frac{a}{2\gamma} -\frac{\tau}{\gamma^2} {-}\frac{C_1 \eta}{\gamma}}} d\tau\,.
\end{split}
\end{equation}
Now, by definition of $r_\gamma$ and $r_\gamma'$, we know that
$$
c_1 =u_\gamma(r_\gamma)\le  u_\gamma(r_\gamma') \le  \gamma - \frac{\mathcal T_\gamma (r_\gamma')}{\gamma} + \eta \frac{\mathcal{T}_\gamma(r_\gamma')}{\gamma^2} \le  \gamma - \frac{\mathcal T_\gamma (r_{\gamma}')}{\gamma} + C_2 \eta\,.
$$
Since $c_1>\frac{a}{2}$ and $C_1,C_2>0$ do not depend on $\eta$,  we can find $\delta>0$ such that 
$$
1- \frac{a}{2\gamma} -\frac{\tau}{\gamma^2} -\frac{C_1 \eta}{\gamma}  \ge  \bra{c_1-\frac{a}{2}{-}\eta(C_1+C_2)} \frac{1}{\gamma}\ge \frac{\delta}{\gamma}\,,
$$
for any $\tau \le \mathcal T_\gamma (r_{\gamma}')$ and any sufficiently small $\eta$. 
Thus, \eqref{Integral} implies that
\[
\begin{split}
\int_{\rho_\gamma}^{r_\gamma'} \Delta u_\gamma 2\pi r \, dr & \lesssim \frac{1}{\gamma}  \int_{\tau({\rho}_\gamma)}^{\tau({r}_\gamma')} \exp\bra{ -\tau \frac{\delta}{\gamma}} d\tau\,, \\ 
& {\lesssim}~ \frac{1}{\gamma}  \int_{\frac{\gamma^2}{2}}^{\infty} \exp\bra{ -\tau \frac{\delta}{\gamma}} d\tau \,,
\\ & = O\left( \exp\left(-\frac{\delta}{2}\gamma\right)\right),
\end{split}
\]
where we have also used that $\mathcal{T}_\gamma (\rho_\gamma)= \frac{\alpha_\gamma \gamma^2}{2}\ge \frac{\gamma^2}{2}$. This completes the proof of \eqref{Claim}.
Now, observe that \eqref{LapUpToRhoGamma} and \eqref{Claim} imply that
$$
\int_{B_0(r)}\Delta u_\gamma \, dy = \int_{B_0(\rho_\gamma)} \Delta u_\gamma \, dy + O\left(  \exp\left(-\frac{\delta}{2} \gamma\right)\right) = \frac{4\pi}{\gamma} \alpha_ \gamma + O\bra{\frac{1}{\gamma^3}},
$$
for any $r\in [\rho_\gamma,r_\gamma']$. Moreover, by \eqref{DefTSGamma} and \eqref{Scales}, we have that
\[
\begin{split}
\int_{B_0(r)} \Delta t_\gamma  \, dy =  - 2\pi r t_\gamma'(r) = -\frac{4\pi \frac{r^2}{\mu_\gamma^2}}{1+\frac{r^2}{\mu_\gamma^2}}= - 4\pi + O\left(\frac{\mu_\gamma^2}{\rho_\gamma^2}\right) = -4\pi +o\left(\frac{1}{\gamma^2}\right).
\end{split}
\]
In particular we find that
$$
u_\gamma'(r) + \frac{\mathcal{T}_\gamma'(r)}{\gamma} =  {-}\frac{1}{2\pi r } \int_{B_0(r)}\Delta \left(u_\gamma+\frac{\mathcal{T}_\gamma}{\gamma}\right) dy = \frac{2}{r} o\left(\frac{1}{\gamma^2}\right), 
$$
for any $r\in [\rho_\gamma,r_\gamma']$. Applying the fundamental theorem of calculus and using \eqref{AtRhoGamma}, we find  that 
$$
u_\gamma(r)- \left(\gamma-\frac{\mathcal{T}_\gamma(r)}{\gamma}\right) =  u_\gamma(\rho_\gamma) -\left( \gamma - \frac{\mathcal{T}_\gamma(\rho_\gamma)}{\gamma} \right) +  o\left(\frac{1}{\gamma^2}\right)\log \frac{r^2}{\rho_\gamma^2} = o\left(\frac{\mathcal{T}_\gamma(r)}{\gamma^2}\right). 
$$
Then we must have $r_\gamma = r_\gamma'$ for any large $\gamma$ (and in particular \eqref{LapInter} follows from \eqref{Claim}). Since $\eta$ can be arbitrarily  small, we get \eqref{SecondExpansion}.  
\end{proof}

\begin{Step}\label{Thm3StepLimit}
As $\gamma\to +\infty$ we have $\beta_\gamma \to  \beta_\frac{a}{2}>0$ and   $u_\gamma \to u_\frac{a}{2}$  in $C^1_{loc}(\bar \Omega\setminus \{0\})$.
\end{Step}
\begin{proof} Let $c_1$, $\rho_\gamma$ and $r_\gamma$ be as in \eqref{G1},  \eqref{DefRhoGammaG}, and  \eqref{DefRGamma}. Since $u_\gamma\le c_1$ in $B_0(1)\setminus B_0(r_\gamma)$,  \eqref{EqLemmaThm3},  \eqref{BoundBetaGamma}, \eqref{LapUpToRhoGamma} and \eqref{LapInter}  give that $\Delta u_\gamma$ is bounded in $L^1(\Omega)$. Hence, we have that 
$$
u_\gamma (r)= O\left(\log \frac{1}{r}  \right)
$$
for $r\in(0,1]$, so that $u_\gamma$ is locally bounded in $\bar \Omega \setminus \{0\}$. By \eqref{EqLemmaThm3}, \eqref{BoundBetaGamma} and elliptic estimates, up to a subsequence we have that $\beta_\gamma\to \beta_\infty \in  [0,+\infty)$ and   $u_\gamma \to u_\infty$ in $C^1_{loc}(\bar \Omega\setminus \{0\})$, where $u_\infty$ solves 
\begin{equation}\label{EqPunctured}
\begin{cases}
\Delta u_\infty = \beta_\infty u_\infty g(u_\infty),\quad \text{ in }  \Omega\setminus \{0\} \,,\\
u_\infty =0, \quad \text { in }\partial \Omega\,.
\end{cases}
\end{equation}
Note that for $r\in [0,r_\gamma]$, we have $u_\gamma(r)\ge c_1$ and 
\[
u_\gamma (r) = c_1 + \int_{r}^{r_\gamma}  \frac{1}{2\pi s} \int_{B_0(s)} \Delta u_\gamma dy ~{ds}\le c_1 +  \frac{1+o(1)}{\gamma} \log\frac{r_\gamma^2}{r^2}\,,
\]
where the last inequality follows from \eqref{LapUpToRhoGamma} and \eqref{LapInter}. Then necessarily $r_\gamma\to 0$, otherwise we would have $u_\infty\equiv c_1$ in  $B_0(\delta)\setminus \{0\}$, for some $\delta>0$, which contradicts \eqref{EqPunctured}. Then, since $u_\gamma\le c_1$ in $[r_\gamma,1]$ implies  $u_\infty\le c_1$ in $(0,1]$, by \eqref{EqPunctured} and standard elliptic regularity we get that $u_\infty\in C^1(\bar \Omega)$ and that $u_\infty$ solves  
\begin{equation}\label{EqDisk}
\begin{cases}
\Delta u_\infty = \beta_\infty u_\infty g(u_\infty)\quad \text{ in } \Omega\,,\\
u_\infty = 0 \quad \text{ on }\partial \Omega.
\end{cases}
\end{equation}
Now, let us take a sequence $(\sigma_\gamma)_\gamma$ such that $r_\gamma\le \sigma_\gamma\to 0$, 
\begin{equation}\label{Sigma1}
\|u_\gamma-u_\infty\|_{C^1(\bar \Omega \setminus B_0(\sigma_\gamma))}\to 0\,,
\end{equation}
and 
\begin{equation}\label{Sigma2}
\frac{1}{\gamma}\log\frac{1}{\sigma_\gamma^2} =o(1)\,.
\end{equation}  
Applying the fundamental therorem of calculus and using \eqref{LapUpToRhoGamma}, \eqref{LapInter}, \eqref{Sigma1} and \eqref{Sigma2} we get that 
\begin{equation}\label{Exp2}
\begin{split}
u_\gamma (r_\gamma) &=  u_\gamma(\sigma_\gamma) + \int_{r_\gamma}^{\sigma_\gamma} \frac{1}{2\pi r}\int_{B_0(r)} \Delta u_\gamma  dy ~{dr}\,,\\
&= u_\gamma(\sigma_\gamma) +  \int_{r_\gamma}^{\sigma_\gamma} \frac{1}{r} \left( \frac{2+o(1)}{\gamma} + O(r^2 )\right)  dr\,, \\
& = u_\gamma(\sigma_\gamma)+ \frac{1+o(1)}{\gamma} \log\frac{\sigma_\gamma^2}{r_\gamma^2} + O(\sigma_\gamma^2)\,,\\
& = u_\infty(0) + \frac{1+o(1)}{\gamma} \log\left(\frac{1}{r_\gamma^2}\right) + o(1)\,.
\end{split}
\end{equation}
Note that, since $u_\gamma(r_\gamma)=c_1$, one has necessarily that 
\begin{equation}\label{RoughBound}
\frac{1}{\gamma} \log\left(\frac{1}{r_\gamma^2}\right)=O(1)\,.
\end{equation}
By Step \ref{LemmaImprovedExpansion} we can compute $u_\gamma(r_\gamma)$ according to the expansion in \eqref{SecondExpansion} and find 
\begin{equation}\label{Exp1}
\begin{split}
u_\gamma (r_\gamma) & = \gamma - \left(1+\frac{a}{2\gamma} \right) \frac{ t_\gamma(r_\gamma)}{\gamma} +o(1)\,, \\
 &= \gamma - \left(1+\frac{a}{2\gamma} \right)\left(\gamma-a-\frac{1}{\gamma}\log \left(\frac{1}{r_\gamma^2}\right) - \frac{1}{\gamma}\log \left(\frac{1}{\beta_\gamma}\right)+o(1)\right)\,,\\
& = \frac{a}{2}+ \frac{1+o(1)}{\gamma}\log \left(\frac{1}{r_\gamma^2}\right) + \frac{1+o(1)}{\gamma}\log \left(\frac{1}{\beta_\gamma}\right) +o(1)\,.
\end{split}
\end{equation}
Then, comparing \eqref{Exp1} with \eqref{Exp2} and using \eqref{RoughBound}, we find that 
\begin{equation}\label{Almost}
u_\infty(0) = \frac{a}{2} + \frac{1+o(1)}{\gamma}\log\left( \frac{1}{\beta_\gamma}\right) + o(1)\,. 
\end{equation}
Note that one cannot have $\beta_\infty =0$, otherwise \eqref{EqDisk} would imply $u_\infty\equiv 0$ in $B_1(0)$ and in particular $u_\infty(0) = 0$, which contradicts \eqref{Almost}. Then, we have $\beta_\infty >0$, so that \eqref{Almost} implies  $u_\infty(0)=\frac{a}{2}$. The uniqueness result of Lemma \ref{LemmaExistNew} implies that $\beta_\infty=\beta_\frac{a}{2}$ and $u_\infty=u_\frac{a}{2}$. 
\end{proof}

In view of Step \ref{Thm3StepLimit}, in order to prove Theorem \ref{Thm3} it remains to prove that the quantification in \eqref{QuantifGa} holds true. Indeed, at that stage, this also implies that $u_\gamma\to u_\frac{a}{2}$ weakly in $H^1_0(\Omega)$, as claimed in \eqref{WLimThm3}.  By Step \ref{Thm3StepLimit}, we can find $(\sigma_\gamma)_{\gamma}$ such that  \[
\|u_\gamma -u_\frac{a}{2}\|_{C^1(B_0(1)\setminus B_0(\sigma_\gamma))}\to 0 \quad \text{ and } r_\gamma \le \sigma_\gamma\to 0\,,
\] 
as $\gamma\to +\infty$, where $r_\gamma$ is as in \eqref{DefRGamma}.  Then, by \eqref{EqLemmaThm3} we have
\begin{equation}\label{EqLast2}
\begin{split}
\int_{B_0(1)\setminus B_0(\sigma_\gamma)} u_\gamma \Delta u_\gamma dx  & =  \int_{B_0(1)} \beta_\frac{a}{2} u_\frac{a}{2}^2 g(u_\frac{a}{2}) dx + o(1)\,, \\ 
& = \int_{B_0(1)} |\nabla u_\frac{a}{2}|^2 dx + o(1)\,,
\end{split}
\end{equation}
and, since $u_\gamma\le c_1$ in $\Omega\setminus B_0(r_\gamma)$, that
\begin{equation}\label{EqLast3}
\int_{B_0(\sigma_\gamma)\setminus B_{0}(r_\gamma)} u_\gamma \Delta u_\gamma dx = O(\sigma_\gamma^2)\,.
\end{equation}
Finally, \eqref{ExpEnergy1} and \eqref{LapInter} with $u_\gamma\le \gamma$ give
\begin{equation}\label{EqLast4}
\int_{B_0(r_\gamma)}  u_\gamma \Delta u_\gamma dx =  4\pi + o(1)\,. 
\end{equation}
Clearly \eqref{EqLast2}, \eqref{EqLast3} and \eqref{EqLast4} give \eqref{QuantifGa} after an integration by parts. 


\end{document}